\documentclass[reqno, 11pt,epsfig,amsfonts]{amsart}
\usepackage{amssymb, amsmath, amsthm}

\usepackage[latin1]{inputenc} % 必须加这个，否则投稿编译不成功
\usepackage{listings, tcolorbox}% 必须加这个，否则投稿编译不成功

\usepackage{tocvsec2}
\usepackage{color}
\usepackage{epsfig}
\usepackage{amsmath}
\usepackage{amssymb}
\usepackage{amscd}
\usepackage{graphicx}
\usepackage{hyperref}
\usepackage{cleveref}
\usepackage{subfigure}
\usepackage{mathrsfs}
\usepackage[T1]{fontenc}
\usepackage{tikz}
\usepackage{floatrow}
\usepackage{float}
\usepackage{pifont,bm}
\usepackage{indentfirst}
\usepackage{enumerate}
\usepackage{amsmath}

\numberwithin{equation}{section}
\newtheorem{lemma}{Lemma}[section]
\newtheorem{theorem}{Theorem}
\newtheorem{re}{Remark}[section]

\newtheorem{prop}{Proposition}[section]

   \newtheorem{definition}{Definition}[section]

\usepackage{epsfig}

\numberwithin{equation}{section}
\numberwithin{theorem}{section}
\numberwithin{prop}{section}
\numberwithin{lemma}{section}
\numberwithin{re}{section}
\numberwithin{coro}{section}
\theoremstyle{definition}

\subjclass[2020]{37K10, 35Q35, 35Q51}

\keywords{Integrable system, modified Camassa--Holm equation, Stability, $N$-soliton solutions, Recursion operator}

\thanks{*Corresponding author (sftian@cumt.edu.cn, shoufu2006@126.com).}
\thanks{$^\dagger$Contributed equally as the first author.}
\begin{document}

\title[Stability of $N$-solitons for the mCH equation]{Stability of $N$-soliton solutions for the modified Camassa--Holm equation}

%\today

\author[Lu]{Zhen Lu}
\author[Tang]{Fei Tang$^{\dagger}$}
\author[Tian]{Shou-Fu Tian$^{*,\dagger}$}
\address{Zhen Lu, Fei Tang, Shou-Fu Tian (Corresponding author)\newline
School of Mathematics, China University of Mining and Technology, Xuzhou 221116, People's Republic of China}
\email{sftian@cumt.edu.cn, shoufu2006@126.com}

\begin{abstract}
In this work, we address the stability of $N$-soliton solutions to the completely integrable modified Camassa--Holm (mCH) equation. Recently, Li, Liu, and Zhu (Math. Ann. 392 (2025), 899--932) established the orbital stability of 2-soliton solutions in $H^4(\mathbb{R})$ with respect to the solution $u$ and highlighted the stability of mCH $N$-soliton solutions remains an urgent challenge. Motivated by their work, we systematically investigate the stability of mCH $N$-solitons. We first employ the bi-Hamiltonian structure of mCH to construct a novel hierarchy of explicit conservation laws with well-defined regularity domains. Then by formulating an appropriate Lyapunov functional, we apply the Inverse Scattering Transform to conduct a rigorous spectral analysis on the recursion operators. Finally, we demonstrate that the mCH $N$-solitons are non-isolated constrained minimizers of a variational problem. Our analysis proves that the $N$-soliton solutions of the mCH equation are both dynamically and orbitally stable in $H^{N+1}(\mathbb{R})$. Notably, when reduced to the 2-soliton case, our framework establishes stability in $H^3(\mathbb{R})$, which improves upon the existing regularity threshold.
\end{abstract}

\maketitle
%\tableofcontents

\section{Introduction}\label{sec.1}
In this work, we consider the stability of $N$-soliton solutions for the modified Camassa--Holm (mCH) equation \cite{ABFS, Fo, Fu}
\begin{align}\label{mch}
	{m}_{t} + {\left( ( {u}^{2} - {u}_{x}^{2}) m\right) }_{x} + \gamma {u}_{x} = 0,\quad t > 0,\quad x \in \mathbb{R},
\end{align}
where $u(t,x)$ is a real-valued function of the time-space variables $(t,x)$, the momentum density is defined as $m(t,x)=(1-\partial_x^2)u(t,x)$, and the constant $\gamma>0$ acts as the linear dispersion parameter.

The mCH equation \eqref{mch} was originally identified as a completely integrable system by Fuchssteiner via the method of recursion operators \cite{Fu,Qiao}. It was independently recovered by Olver and Rosenau through the application of tri-Hamiltonian duality within the bi-Hamiltonian framework of the modified Korteweg-de Vries equation \cite{OR}. From a physical standpoint, the mCH equation \eqref{mch} serves as a model for the unidirectional propagation of moderate-amplitude shallow-water waves over a flat bottom \cite{CHL}. Here, $u(t,x)$ denotes the horizontal fluid velocity at a specific depth, and the linear dispersion parameter $\gamma > 0$ corresponds to the critical shallow-water wave speed.

Moreover, the mCH equation \eqref{mch} is deeply intertwined with other prominent nonlinear dispersive models. Notably, under a characteristic coordinate transformation, it reduces to the integrable short-pulse equation \cite{SW}, expressed as
\begin{align}\label{spe}
	u_{xt} = u + \frac{1}{6}(u^3)_{xx}.
\end{align}
Given that equation \eqref{spe} is a paradigmatic model for the propagation of ultra-short optical pulses in silica fibers, this connection further underscores the broad physical relevance and intricate mathematical structure of the mCH equation.

A hallmark of the mCH equation \eqref{mch} is its complete integrability, which is intrinsically tied to its bi-Hamiltonian structure \cite{FF,GLO}. Specifically, the equation admits two compatible Hamiltonian operators, enabling the formulation
\begin{align}\label{shmdjg}
	{m}_{t} = \mathcal{K}_{0}\frac{\delta E}{\delta m} = \mathcal{K}_{1}\frac{\delta H}{\delta m},
\end{align}
where $m = u - {u}_{xx}$. The two compatible Hamiltonian operators $\mathcal{K}_0$ and $\mathcal{K}_1$ are given by
\begin{align}
	\mathcal{K}_{0}=-{\partial}_{x}m{\partial}_{x}^{-1}m{\partial}_{x}-\frac{1}{2}\gamma{\partial}_{x}, \quad \mathcal{K}_{1} = {\partial }_{x} - {\partial }_{x}^{3},
\end{align}
which correspond to the Hamiltonian functionals
\begin{align}\label{E}
	&E(u)=\int_{\mathbb{R}}\left({{u}^{2} + {u}_{x}^{2}}\right) dx,\\
	&H(u)= \frac{1}{4}\int_{\mathbb{R}}\left( {-{u}^{4} - 2{u}^{2}{u}_{x}^{2} + \frac{1}{3}{u}_{x}^{4} - {2\gamma }{u}^{2}}\right) dx.\label{H}
\end{align}

Identifying the natural domains for these conserved quantities is essential. The functional $E(u)$ corresponds to the squared norm of the standard Sobolev space $H^1(\mathbb{R})$, ensuring it is well-defined for any velocity profile $u \in H^1(\mathbb{R})$. In contrast, the Hamiltonian $H(u)$ contains a higher-degree nonlinear term, namely $\frac{1}{3}u_x^4$. As a result, $H(u)$ requires higher regularity and is rigorously defined only within $W^{1,4}(\mathbb{R})$. Consequently, the natural admissible space for the full bi-Hamiltonian system is restricted to the intersection $H^1(\mathbb{R}) \cap W^{1,4}(\mathbb{R})$ \cite{LL}.

Moreover, by employing the operator $\mathcal{K}_1$ and the identity $m = (1-\partial_x^2)u$, the bi-Hamiltonian system \eqref{shmdjg} admits an equivalent, more direct representation in terms of the primary velocity variable $u$, given by
\begin{align}\label{hmdjgu}
	\frac{\partial u}{\partial t} = \mathcal{J}\frac{\delta H}{\delta u},
\end{align}
where the operator $\mathcal{J}={\partial}_{x}{\left(1-{\partial}_{x}^{2}\right) }^{-1}$ is skew-symmetric and is a bounded operator on ${L}^{2}(\mathbb{R})$. This specific formulation is instrumental for analyzing the underlying dynamics and variational properties of the solutions \cite{CH,GLO,WL,XZQ}.

Beyond its elegant algebraic framework, the mCH equation \eqref{mch} exhibits a rich spectrum of nonlinear wave phenomena, which notably include the existence of solitary waves and multi-soliton solutions \cite{LL,LLZ1,LLZ2}. Rigorous analysis establishes that the mCH equation admits smooth solitary wave solutions if and only if the wave propagation speed $c$ satisfies the condition $0 < \gamma < c \le 2\gamma$. These traveling wave states assume the standard ansatz $u(t,x) = \phi_c(x - ct + x_0)$, where $\phi_c$ denotes the localized wave profile and $x_0 \in \mathbb{R}$ is an arbitrary phase shift. Substituting this traveling wave ansatz into the mCH equation \eqref{mch} yields a nonlinear ordinary differential equation governing the profile function:
\begin{align}\label{profile_eq1}
	\left(\phi_c^2 - (\partial_\xi\phi_c)^2 - c\right)\left(\phi_c - \partial_\xi^2\phi_c\right) = -\gamma\phi_c,
\end{align}
where $\xi = x - ct + x_0$. Multiplying \eqref{profile_eq1} by $\partial_\xi\phi_c$ and integrating once with respect to $\xi$ yields the corresponding first-order integrated form
\begin{align}\label{profile_eq2}
	\left(\phi_c^2 - (\partial_\xi\phi_c)^2 - c\right)^2 = c^2 - 2\gamma\phi_c^2,
\end{align}
which characterizes the homoclinic orbit associated with the smooth solitary wave dynamics.

In addition to single solitary waves, the complete integrability of the mCH equation ensures the existence of complex multisoliton solutions $U^{(N)}(t,x)$, which describe the nonlinear interaction of $N$ distinct solitons \cite{Ma}. A key dynamical feature of these configurations is their purely elastic interaction \cite{Lax}. As time approaches infinity, the composite solution asymptotically decouples into a linear superposition of $N$ independent and spatially isolated solitary waves traveling at distinct speeds $c_i$ for $i = 1, 2, \dots, N$, namely
\begin{align}\label{asymptotic_N}
	U^{(N)}(t,x) \sim \sum_{i=1}^{N} \phi_{c_i}(x - c_i t - x_i), \quad \text{as } t \to \infty,
\end{align}
where the phase parameters $x_i \in \mathbb{R}$ depend on the corresponding wave speeds $c_i$. This asymptotic decoupling demonstrates the absence of dispersive effects at spatial infinity, establishing a rigorous analytical foundation for investigating the invariant quantities and stability properties of multisoliton dynamics \cite{TT,YTL}.

The Cauchy problem for the mCH equation has been extensively investigated concerning its well-posedness and wave-breaking phenomena. Initial studies established local well-posedness for initial data in the Sobolev space $H^s(\mathbb{R})$ with $s > \frac{7}{2}$ \cite{MZ}, a regularity threshold that was subsequently relaxed to $s > \frac{5}{2}$ \cite{CLZ,GLO}. Analogous to the classical CH model \cite{Co,CE}, the mCH equation exhibits finite-time wave-breaking, characterized by a uniformly bounded solution with an unbounded spatial derivative \cite{GLO}. Conversely, global existence is guaranteed for sufficiently regular initial data ($s \ge 3$) that satisfy specific momentum threshold conditions, effectively precluding wave-breaking \cite{LLZ1}.

The stability analysis of solitary waves and multisolitons remains an active area of research in nonlinear dispersive equations.
Comparing the mCH equation \eqref{mch} with the classical Korteweg-de Vries (KdV) \cite{KdV1895,J} and Camassa--Holm (CH) \cite{CH, CL09} models provides an essential methodological benchmark. These fundamental shallow water equations capture different nonlinear dispersive regimes. Extensive investigations into KdV multisoliton stability culminated in the seminal constrained variational principle developed by Maddocks and Sachs \cite{MS} in the natural energy space, a result recently extended to the negative Sobolev space $H^{-1}(\mathbb{R})$ by Killip and Vi\c{s}an \cite{KV22}. Concurrently, CH theory established the orbital stability of single solitons and peakons in $H^1(\mathbb{R})$ through spectral methods \cite{CM01,CS02,CS00}, while multisoliton configurations were subsequently proven stable in $H^n(\mathbb{R})$ by adapting localized spectral techniques \cite{CGI07,EM07,WL}. Furthermore, Li and Zhang \cite{LZ} utilized nonlocal conservation laws to prove that the $N$-solitons of the CH equation are orbitally stable. Additionally, Wu and Tian \cite{WTLW} demonstrated that the $N$-solitons of the two-component CH equation are dynamically stable and that its 2-soliton solutions are orbitally stable. These fundamental advancements highlight the pivotal role of higher-order conserved quantities in rigorously establishing the stability of multisolitons.

For the mCH equation, initial investigations established the orbital stability of the smooth single soliton through classical constrained minimization \cite{LL}. Utilizing the foundational conserved quantities $E(m)$ and $H(m)$, stability was proven within the intersection space $W^{1,4}(\mathbb{R}) \cap H^1(\mathbb{R})$ for the momentum density $m = u - u_{xx}$. To circumvent analytical obstacles introduced by the higher-degree nonlinearities in $H(u)$, subsequent studies shifted to the momentum density formulation $m$. Leveraging the locally conserved functionals
\begin{align}
	&F_1(m) = \int_{\mathbb{R}}\left(\frac{-2}{\gamma}M+\sqrt{2}\gamma^{-\frac{1}{2}}\right)dx,\\
	&F_2(m) = \int_{\mathbb{R}}\left(\frac{m_x^2}{2M^5}-\frac{1}{\gamma M}+\sqrt{2}\gamma^{-\frac{3}{2}}\right)dx,
\end{align}
where $M=\sqrt{m^2+\gamma/2}$, researchers successfully demonstrated the orbital stability of the single smooth solitary wave in $H^3(\mathbb{R})$ \cite{LLZ1}.

Recently, this momentum-driven framework was extended to investigate double soliton solutions \cite{LLZ2}. Utilizing the multiplier method, a higher-order conserved integral
\begin{align}
	F_3(m) = \int_{\mathbb{R}}\left(\frac{m_{xx}^2}{2M^7}+\frac{5m_x^2}{4M^7}-\frac{7m_x^4}{2M^9}+\frac{35\gamma m_x^4}{16M^{11}}-\frac{1}{4\gamma M^3}-\frac{2}{\gamma^2 M}+\frac{5\sqrt{2}}{2}\gamma^{-\frac{5}{2}}\right)dx,
\end{align}
was rigorously derived in \cite{CLZ,FF,Ma}. The subsequent stability analysis for the $2$-soliton solutions is extended to the higher-regularity Sobolev space $H^4(\mathbb{R})$ for the primary velocity variable $u$.

Despite these significant advancements, the study of the mCH equation requires further investigation to relax existing regularity constraints and to extend the stability framework to multi-soliton configurations. The primary objective of this work is to establish the dynamical and orbital stability of $N$-soliton solutions in $H^{N+1}(\mathbb{R})$.
Driven by the unique cubic nonlinearities and bi-Hamiltonian structure of the mCH equation, this study addresses three central issues: utilizing its bi-Hamiltonian hierarchy to generate conservation laws for stability analysis in appropriate functional spaces, establishing a rigorous mathematical framework to determine the spectral inertia of generalized $N$-solitons, and applying spectral decomposition and constrained variational principles to verify the dynamical and orbital stability of mCH $N$-solitons under complex geometric constraints.

Our strategy is to adapt and generalize the framework of Maddocks and Sachs \cite{MS}, which originally addressed the stability of multi-solitons for the KdV equation via a constrained variational principle. The proof of stability is thus reduced to analyzing the spectral inertia of the second variation of a specially constructed Lyapunov functional. This approach has been extended to various integrable models \cite{Lco,LT,NL,WA,WL,WT,WTLW,XGW}. However, the application of this methodology to the mCH $N$-solitons presents several significant technical challenges.

Our analysis addresses three principal technical challenges. The first involves the domain characterization of the conservation hierarchy, as higher-order conserved quantities $\{F_n\}$ often involve higher-order spatial derivatives such as $m_{xx}$, necessitating restrictive functional topologies like $H^4(\mathbb{R})$. Identifying a recursive set of conserved quantities that balances mathematical tractability with minimal regularity requirements is therefore essential. The second challenge lies in the spectral analysis of the linearized operator, which for general mCH $N$-solitons escalates to a differential operator of order $2(N+1)$. The presence of these high-order derivative terms and non-constant coefficients exhibiting complex nonlinear interactions precludes the specific algebraic factorizations used in two-soliton scenarios, thereby demanding a fully systematic framework to determine the spectral inertia. Finally, the multidimensional stability constraints induced by $N$-soliton interactions pose a formidable structural hurdle. Due to the inherently nonlocal nature of the mCH recursion operator, the constrained variational framework is highly sensitive to the generalized conservation laws and their corresponding Lagrange multipliers, which rigorously requires the explicit evaluation of a highly coupled $N \times N$ Hessian matrix governed by generalized symmetric polynomials.

In order to systematically overcome the aforementioned difficulties, our analysis employs the following strategies.

To address the initial challenge concerning the domain of conservation laws and stringent regularity constraints, we leverage the intrinsic bi-Hamiltonian structure, circumventing the limitations of the conventional $\{F_n\}$ sequence. By combining the compatible Hamiltonian operators $\mathcal{K}_0$ and $\mathcal{K}_1$, we define the recursion operator

\begin{align}
\mathcal{K}[m] = \mathcal{K}_1^{-1}\mathcal{K}_0 = -(1-\partial_x^2)^{-1}\left(m\partial_x^{-1}m\partial_x + \frac{\gamma}{2}\right),
\end{align}

and the explicit inverse of this operator is given by

\begin{align}\label{hdgszm}
\widehat{\mathcal{K}}[m] = -\frac{2}{\gamma}\left(1 - \frac{m}{M}\partial_x^{-1}\frac{m}{M}\partial_x\right)\left(1-\partial_x^2\right),
\end{align}

where $M = \sqrt{m^2 + \frac{\gamma}{2}}$. Applying $\widehat{\mathcal{K}}[m]$ recursively from the foundational energy $E_0(u) = E(u)$ generates a novel hierarchy of conserved quantities $\{E_n\}_{n \ge 0}$ satisfying the recursive relation

\begin{align}\label{dggxm}
\frac{\delta E_n}{\delta m} = \widehat{\mathcal{K}}[m]\frac{\delta E_{n-1}}{\delta m}, \quad n = 1, 2, \dots, N.
\end{align}

This recursive scheme yields explicit higher-order functionals, for instance,

\begin{align}
E_1(m) &= \int_{\mathbb{R}}\left(-2\sqrt{\frac{2}{\gamma}}M+2\right)dx,\\
E_2(m) &= \int_{\mathbb{R}}\left(\sqrt{\frac{\gamma}{2}}\frac{m_x^2}{M^5}+\frac{2}{\gamma}\sqrt{\frac{2}{\gamma}}\frac{m^2}{M}\right)dx.
\end{align}

To perform the variational analysis directly on the primary velocity profile $u$, the recursion relations \eqref{dggxm} are reformulated as

\begin{align}
\frac{\delta E_n}{\delta u} = \widehat{\mathcal{R}}(u)\frac{\delta E_{n-1}}{\delta u}, \quad n = 1, 2, \dots, N.
\end{align}

Here, the transformed recursion operator $\widehat{\mathcal{R}}(u)$ is defined as
\begin{align}
\widehat{\mathcal{R}}(u) = -\frac{2}{\gamma}\left(1-\partial_x^2\right)\left(1 - \frac{m}{M}\partial_x^{-1}\frac{m}{M}\partial_x\right),
\end{align}
with its inverse
\begin{align*}
\mathcal{R}(u) = (1-\partial_x^2)\mathcal{K}[m](1-\partial_x^2)^{-1} = -\left(m\partial_x^{-1}m\partial_x + \frac{\gamma}{2}\right)(1-\partial_x^2)^{-1}.
\end{align*}

Notably, the natural domains for the generated sequence $\{E_n\}_{n=0}^N$ correspond precisely to the Sobolev spaces $H^{n+1}(\mathbb{R})$ (as proved in Section \ref{sec.2}). This property directly fulfills our objective to lower the regularity threshold and identifies the optimal energy space for the stability analysis.

\begin{re}
By exploiting the recursion operator $\widehat{\mathcal{R}}(u)$ to generate the $\{E_n\}$ hierarchy, the highest spatial derivative in $E_N(u)$ is bounded by order $N+1$. This reduction in regularity requirements allows us to rigorously establish the dynamical stability of $N$-soliton solutions in the optimal, lower-regularity space $H^{N+1}(\mathbb{R})$.
\end{re}

Equipped with this functional hierarchy, we apply the constrained variational principle by constructing the augmented Lyapunov functional $\mathcal{S}_{N}$ for the mCH $N$-solitons as
\begin{align}
\mathcal{S}_{N}(u) = (-1)^N\left(E_{N}(u) + \sum_{m=0}^{N-1} \mu_m E_m(u)\right),\label{lagrange_N}
\end{align}

where the Lagrange multipliers $\mu_{m}$ ($m=0, 1, \dots, N-1$) correspond to elementary symmetric functions of the wave speeds $c_{1}, \dots, c_{N}$ (as detailed in Section \ref{sec.2}). We demonstrate that $U^{(N)}$ is a critical point of $\mathcal{S}_{N}$, satisfying the generalized Euler-Lagrange equation

\begin{align}
\frac{\delta E_{N}(u)}{\delta u} + \sum_{m=0}^{N-1} \mu_{m} \frac{\delta E_{m}(u)}{\delta u} = 0,\quad \text{at } u = U^{(N)}.\label{1.20_N}
\end{align}

Establishing the dynamical stability of $U^{(N)}$ requires showing that it minimizes $E_{N}$ subject to the constraints $E_{m}(u) = E_{m}\left(U^{(N)}\right)$ for $m = 0, 1, \dots, N-1$. This fundamentally requires the self-adjoint second variation operator,

\begin{align}
\mathcal{L}_N := \mathcal{S}''_{N}(U^{(N)}), \label{VL_N}
\end{align}

to be positive definite under modulations along the constrained directions.

To address the second challenge regarding the spectral analysis of the high-order differential operator $\mathcal{L}_N$, we overcome the inapplicability of classical Sturm-Liouville theory by bridging single-soliton spectral data with the composite $N$-soliton structure. We focus our study on the linear operator $L_n = E''_{n}(\phi_c) + \omega E''_{n-1}(\phi_c)$ for the single soliton $\phi_c$ and determine its spectral properties by leveraging the Inverse Scattering Transform (IST) framework \cite{AKNS,LTY22a,LTY22b,OTW,XZQ,YTL}. By constructing squared eigenfunctions $F^\pm(x,k)$ from the Jost solutions of the Lax pair, we formulate a generalized basis that facilitates a rigorous analysis of $L_n$ and its conjugate forms $\mathcal{J}L_n$ and $L_n\mathcal{J}$. Specifically, we exploit the fundamental operator identities establishing that $\mathcal{J}L_n$ and $L_n\mathcal{J}$ commute with the adjoint recursion operator $\widehat{\mathcal{R}}^*(\phi_c)$ and the recursion operator $\widehat{\mathcal{R}}(\phi_c)$, respectively. As $t \to \infty$, the spectrum of the $N$-soliton operator $\mathcal{L}_N$ converges to the union of the individual single-soliton spectra, allowing us to precisely calculate the inertia index of $\mathcal{L}_N$ to verify the critical spectral condition $\mathrm{in}(\mathcal{L}_N) = p(D_N)$. In this formulation, $\mathrm{in}(\mathcal{L}_N)$ denotes the number of negative eigenvalues of $\mathcal{L}_N$ and $p(D_N)$ represents the number of positive eigenvalues of the $N \times N$ Hessian matrix $D_N := \left\{ \frac{\partial^{2} \mathcal{S}_{N}}{\partial \mu_{i} \partial \mu_{j}} \right\}$.

Finally, to resolve the structural complexity of the linearized $N$-soliton operator, we establish dynamical stability via a refined coercivity estimate. Spectral analysis reveals that $\mathcal{L}_N$ possesses exactly $\left[\frac{N+1}{2}\right]$ negative eigenvalues and an $N$-dimensional generalized kernel, corresponding to the potential instability directions associated with spatial translations and scalings. Given that $\mathcal{L}_N$ is a strongly elliptic operator of order $2(N+1)$, we apply a generalized G{\aa}rding's inequality to obtain a local coercivity bound near the $N$-soliton manifold in $H^{N+1}(\mathbb{R})$. By introducing $N$ dynamic temporal modulation parameters to continuously project out the neutral translational directions, we extend this local coercivity to a global bound on the $H^{N+1}$ distance. A standard continuity argument then completes the proof of orbital stability.

Next, we will present our main results as follows. An essential prerequisite for establishing the stability of solitary waves is the existence of global solutions, as demonstrated in \cite{GLO}.
\begin{prop}[Gui, Liu, Olver and Qu \cite{GLO}]
	Given an initial profile $u_0\in H^s(\mathbb{R})$ with $s>\frac{5}{2}$, there exists a maximal time $T>0$, which is independent of the choice of $s$, and a unique solution $u(t)$ to the Cauchy problem associated with the mCH equation \eqref{mch} with initial value $u(0)=u_0$ satisfying
	\begin{align}
		u = u(\cdot; u_0) \in C([0, T); H^s(\mathbb{R})) \cap C^1([0, T); H^{s-1}(\mathbb{R})).
	\end{align}
\end{prop}

Building upon this well-posedness result, we state our main theorems concerning the stability of the $N$-solitons for the mCH equation \eqref{mch}. A central objective of this work is not only to establish the dynamical and orbital stability of the mCH $N$-soliton solutions, but also to further reduce the required regularity of the stability space in \cite{LLZ2}. This reduction relies on the fact that the novel hierarchy of conservation laws $\{E_k\}_{k=0}^{N}$, derived from the bi-Hamiltonian system of the mCH equation, possesses well-defined and optimal regularity domains. Therefore, we first establish the following foundational theorem.

\begin{theorem}\label{thm:domain_conservation_laws}
	For the mCH equation \eqref{mch}, the $k$-th order conservation law $E_k(u) = \int_{\mathbb{R}} \mathcal{E}_k dx$ generated by the recursion operator $\widehat{\mathcal{R}}(u)$ explicitly depends on the spatial derivatives of $u$ up to order $k+1$. Furthermore, the natural Sobolev space for the functional $E_k(u)$ to be well-defined and finite is precisely $H^{k+1}(\mathbb{R})$.
\end{theorem}

Equipped with the aforementioned hierarchy of conservation laws, we are able to construct a suitable augmented Lyapunov functional framed within an optimal energy space. Utilizing the generalized constrained variational framework developed by Maddocks and Sachs \cite{MS}, we establish the following theorem regarding the dynamical stability of the mCH $N$-solitons.

\begin{theorem}[Dynamical stability of $N$-solitons]\label{th1.1}
	Let $U^{(N)}(t,x ; \mathbf{c}, \mathbf{x}_0)$ be the smooth $N$-soliton solution of the mCH equation \eqref{mch} characterized by the wave speed vector $\mathbf{c}=(c_1, \dots, c_N)$ satisfying $\gamma < c_1 < \dots < c_N \le 2\gamma$ and initial phase shifts $\mathbf{x}_0 \in \mathbb{R}^N$. For any $\varepsilon > 0$, there exists a $\delta > 0$ such that if the initial data $u_0 \in H^{N+1}(\mathbb{R})$ satisfies
	\begin{align*}
		\left\|u_{0}(\cdot) - U^{(N)}(0, \cdot \, ; \mathbf{c}, \mathbf{x}_0)\right\|_{H^{N+1}(\mathbb{R})} < \delta,
	\end{align*}
	then the corresponding solution $u(t,x)$ of the mCH equation \eqref{mch} with the initial data $u(0) = u_0$ satisfies $u \in C([0, +\infty); H^{N+1}(\mathbb{R}))$, and for all $t > 0$,
	\begin{align*}
		\inf_{\psi \in G_{\mathbf{c}}} \left\|u(t, \cdot) - \psi(t, \cdot)\right\|_{H^{N+1}(\mathbb{R})} < \varepsilon,
	\end{align*}
	where $G_{\mathbf{c}}$ denotes the manifold generated by the conserved quantities, defined as
	\begin{align*}
		G_{\mathbf{c}} = \left\{ \psi \in H^{N+1}(\mathbb{R}) \mid E_k(\psi) = E_k(U^{(N)}(0, \cdot \, ; \mathbf{c}, \mathbf{x}_0)) \text{ for } k = 0, 1, \dots, N \right\}.
	\end{align*}
\end{theorem}

While Theorem \ref{th1.1} establishes the stability of the entire $N$-soliton manifold $G_{\mathbf{c}}$ under the mCH flow, it is fundamentally important to characterize the stability of the multi-soliton profile up to dynamic spatial shifts. By introducing appropriate temporal modulation parameters to continuously project the perturbed solution onto the translation directions, we can translate the dynamical stability of the manifold into the orbital stability of the wave profile. This leads to our second main result.

\begin{theorem}[Orbital stability of $N$-solitons]\label{th1.2}
	Let $U^{(N)}(t,x ; \mathbf{c}, \mathbf{x}_0)$ be the smooth $N$-soliton solution of the mCH equation \eqref{mch} characterized by the wave speed vector $\mathbf{c}=(c_1, c_2, \dots, c_N)$ satisfying $\gamma < c_1 < c_2 < \dots < c_N \le 2\gamma$. For any $\varepsilon > 0$, there exists a $\delta > 0$ such that if the initial data $u_{0} \in H^{N+1}(\mathbb{R})$ satisfies
	\begin{align*}
		\left\|u_{0}(\cdot) - U^{(N)}(0, \cdot \, ; \mathbf{c}, \mathbf{x}_0)\right\|_{H^{N+1}(\mathbb{R})} < \delta,
	\end{align*}
	then the corresponding solution $u(t,x)$ to the mCH equation \eqref{mch} with the initial condition $u(0, \cdot) = u_{0}(\cdot)$ exists globally in time, i.e., $u \in C([0, \infty); H^{N+1}(\mathbb{R}))$. Furthermore, this solution is orbitally stable, in the sense that there exist continuous spatial shift functions $\mathbf{r}(t) = (r_1(t), r_2(t), \dots, r_N(t))$ such that
	\begin{align*}
		\sup_{t \ge 0}\left\|u(t, \cdot) - U^{(N)}(t, \cdot \, ; \mathbf{c}, \mathbf{r}(t))\right\|_{H^{N+1}(\mathbb{R})} < \varepsilon.
	\end{align*}
\end{theorem}

\begin{re}
Notably, for the case $N=2$, Theorem \ref{th1.2} establishes the orbital stability of mCH 2-solitons in the energy space $H^3(\mathbb{R})$ with respect to the primary velocity variable $u$. This provides a refined regularity threshold compared to the pioneering work of Li, Liu, and Zhu \cite{LLZ2}, where the 2-soliton stability was established in the space $H^2(\mathbb{R})$ for the momentum density $m$ (which is mathematically equivalent to $H^4(\mathbb{R})$ for the primary variable $u$). In \cite{LLZ2}, the stability analysis relies on the conserved quantity $F_3(m)$, where the explicit presence of the $m_{xx}$ term (corresponding to $\partial_x^4 u$) inherently dictates $H^4(\mathbb{R})$ as the requisite admissible space for $u$.
\end{re}

\begin{re}
Our framework leverages the intrinsic bi-Hamiltonian recursion to generate the hierarchy $\{E_k(u)\}$. According to the domain characterization in Theorem \ref{thm:domain_conservation_laws}, the functional $E_2(u)$ used for the 2-soliton case explicitly involves spatial derivatives of $u$ up to the third order. Consequently, by virtue of the general stability criteria in Theorem \ref{th1.2}, the mCH 2-soliton stability can be established in the lower-regularity space $H^3(\mathbb{R})$ with respect to the primary velocity variable $u$. This result complements the foundational advancements in \cite{LLZ2} by relaxing the initial data constraints on $u$ while offering a unified methodology for general $N$-soliton stability.
\end{re}

Following the statement of the main theorems, we outline the organization of the subsequent proofs. Our general strategy hinges on a detailed spectral analysis of the recursion operators associated with the $N$-soliton solutions $U^{(N)}$. By verifying that the constructed augmented Lyapunov functional $\mathcal{S}_N$ satisfies the Maddocks-Sachs stability criterion, we employ a modulation argument to conclude both dynamical and orbital stability.

The proof transforms the nonlinear stability problem into a spectral analysis of a specific linearized operator, exploiting the mCH equation's bi-Hamiltonian recursion operators to construct a suitable Lyapunov functional. The paper is organized as follows:

In Section \ref{sec.2}, we review the IST framework and the Hamiltonian formulation. We construct the recursion operator $\widehat{\mathcal{R}}$ to generate the conservation laws $\{E_n\}_{n=0}^N$, variationally characterize the $N$-solitons as constrained critical points of $\mathcal{S}_N$, and compute the positive inertia index $p(D_N)$ of the associated Hessian matrix.

Section \ref{sec.3} focuses on the spectral analysis of the high-order linearized operator $\mathcal{L}_N = \mathcal{S}_N''(U^{(N)})$. Utilizing IST squared eigenfunctions and the asymptotic decoupling of $N$-solitons as $t \to \infty$, we determine the spectrum of $\mathcal{L}_N$. This rigorously verifies the critical spectral condition $\mathrm{in}(\mathcal{L}_N) = p(D_N) = \left[ \frac{N+1}{2}\right]$.

In Section \ref{sec.4}, we finalize the stability proofs via constrained variational principles. We establish local coercivity in $H^{N+1}(\mathbb{R})$ by applying a generalized G{\aa}rding's inequality to $\mathcal{L}_N$, and employ a dynamic modulation argument to handle translation invariance, which extends the local bound to a global $H^{N+1}$ estimate and concludes both dynamical and orbital stability.

\section{Preliminaries}\label{sec.2}
In this section, we present fundamental results from the inverse scattering transform
theory for the mCH equation (\ref{mch}). The section is organized as follows.
\subsection{Spectral problem}
As a significant approach for obtaining the orbital stability of the mCH equation, we first introduce the spectral problem of the mCH equation and its inverse scattering transform.

The completely integrable mCH equation \eqref{mch} possesses the following Lax pair
\begin{align}\label{lax}
	\Phi_x(x,t,k)=U(x,t,k)\Phi(x,t,k),\\
	\Phi_t(x,t,k)=V(x,t,k)\Phi(x,t,k),
\end{align}
where $k$ is the spectral parameters, the matrix function $U$ and $V$ satisfy the zero- curvature equation $U_t-V_x+UV-VU=0$ and their can be expressed respectively as
\begin{align}
	&U(x,t,\lambda)=\begin{pmatrix}
		-i\mu & \frac{1}{2}\lambda m(x,t) \\
		-\frac{1}{2}\lambda m(x,t) & i\mu
	\end{pmatrix},\\
	&V(x,t,\lambda)=\begin{pmatrix}
		\frac{2i\mu}{\lambda^2}+i\mu\left(u^2-u_x^2\right) & -\frac{u-2i\mu u_x}{\lambda}-\frac{1}{2}\lambda\left(u^2-u_x^2\right)m \\
		\frac{u+2i\mu u_x}{\lambda}+\frac{1}{2}\lambda\left(u^2-u_x^2\right)m & -\frac{2i\mu}{\lambda^2}-i\mu\left(u^2-u_x^2\right)
	\end{pmatrix},
\end{align}
and
\begin{align}
	\mu(k)=\frac{1}{4}\left(k-\frac{1}{k}\right),\quad \lambda(k)=\frac{1}{\sqrt{2\gamma}}\left(k+\frac{1}{k}\right).
\end{align}
Furthermore, we can easily find that
\begin{align}
	\mu=\mu(\lambda,\gamma)=-\frac{i}{2}\sqrt{1-\frac{1}{2}\gamma\lambda^2}.
\end{align}

We consider solutions $u(t,x)$ with decaying boundary conditions as $|x|\to \infty$ and for the continuous spectrum $\mu(k)\in\mathbb{R}$, the spatial spectral problem \eqref{lax} admits fundamental matrix solutions $\Phi(x,k)$ and $\Psi(x,k)$, satisfying
\begin{align}\label{L2}
	\lim_{x\to-\infty}e^{i\mu y\sigma_3}\Phi(x,k)=\mathbb{I},\quad
	\lim_{x\to+\infty}e^{i\mu y\sigma_3}\Psi(x,k)=\mathbb{I},
\end{align}
where $y=x+\int_{+\infty}^{x}(\sqrt{2m^2/\gamma+1}-1)ds$ and $\sigma_3=\operatorname{diag}\left(1,-1\right)$. These are the so-called Jost solutions. Denote the columns of the Jost solutions by
\begin{align}
	\Phi=[\phi^+,\phi^-]=\begin{pmatrix}
		\phi^+_1 & \phi^-_1 \\
		\phi^+_2 & \phi^-_2
	\end{pmatrix},\quad \Psi=[\psi^-,\psi^+]=\begin{pmatrix}
		\psi^-_1 & \psi^+_1 \\
		\psi^-_2 & \psi^+_2
	\end{pmatrix}.
\end{align}
Using the relation $\mathrm{Im} \mu=\frac{1+|k|^2}{4|k|^2}\mathrm{Im}k$, standard arguments show that $\phi^+(x,\lambda)$ and $\psi^+(x,\lambda)$ are analytic for $\mathrm{Im} k>0$, while $\phi^-(x,\lambda)$ and $\psi^-(x,\lambda)$ are analytic for $\mathrm{Im}k<0$.

Moreover, the smooth solitary wave solution $\phi_c$ of the mCH equation \eqref{mch} was derived by Matsuno \cite{Ma} via the bilinear method. For our purposes, it can be expressed as
\begin{align}
	&u=\frac{16i\mu_1}{\gamma\lambda_1^3}\operatorname{sgn}(\lambda_1)\frac{\cosh(\xi)}{\cosh(2\xi)+\frac{1-4\mu_1^2}{1+4\mu_1^2}},\\
	&\xi=2i\mu_1\left(x-\frac{2t}{\lambda_1^2}-\ln\left(\frac{1+\frac{1+2i\mu_1}{1-2i\mu_1}e^{2\xi}}{1+\frac{1-2i\mu_1}{1+2i\mu_1}e^{2\xi}}\right)\right),
\end{align}
where $\mu_1=\mu(k_1)$, $\lambda_1=\lambda(k_1)$, and the discrete spectral point $k_1$ satisfies $|k_1|=1$ and $a(k_1)=0$. Consequently, one has
$c=\frac{2}{\lambda_1^2}$.

Furthermore, from our other work \cite{TAT}, the smooth solitary-wave solution $\phi_c$ admits the complete set
\begin{align}\label{complete}
		\Bigg\{  \left(1-\partial_x^2\right)\left((\psi^-_{1})^2 + (\psi^-_{2})^2\right)(x,k),\;\text{for } k\in \mathbb{R}^+; \quad &\left(1-\partial_x^2\right)\left((\psi^-_{1})^2 + (\psi^-_{2})^2\right)(x,k_1^-); \notag \\
		& \left(1-\partial_x^2\right)\frac{d}{dk}\left((\psi^-_{1})^2 + (\psi^-_{2})^2\right)(x,k_1^-) \Bigg\}
\end{align}
and the other complete set
\begin{align}
	\bigg\lbrace & \left((\psi^-_{1})^2 + (\psi^-_{2})^2\right)_x(x,k),\;\text{for } k\in \mathbb{R}^+; \notag \\
	& \left((\psi^-_{1})^2 + (\psi^-_{2})^2\right)_x(x,k_1^-);\; \frac{d}{dk}\left((\psi^-_{1})^2 + (\psi^-_{2})^2\right)_x(x,k_1^-) \bigg\rbrace .
\end{align}
where $k_1^-$ is the discrete spectral point on the unit circle in the fourth quadrant corresponding to $\phi_c$. This relation will be helpful for our subsequent research.

\subsection{Hamiltonian formation}
As the foundation for analyzing the multi-soliton stability of the mCH equation (\ref{mch}), it is crucial to determine the domain of the conservation laws so as to define the space of soliton stability. Hence, we will proceed to establish the proof of Theorem \ref{thm:domain_conservation_laws}

To rigorously establish this theorem, we divide the proof into the following lemmas. First, we determine the exact highest order of spatial derivatives in the variational derivative $\frac{\delta E_k}{\delta u}$ and verify its non-degeneracy.
	\begin{lemma}\label{lem:variational_derivative_order}
		For any integer $k \ge 0$, the variational derivative $G_k := \frac{\delta E_k}{\delta u}$ contains spatial derivatives of $u$ up to exactly order $2k+2$. Moreover, the principal symbol of the highest-order differential operator in $G_k$ is strictly positive and bounded away from zero.
	\end{lemma}
	\begin{proof}
		We proceed by mathematical induction on $k$. For the base case $k=0$, the fundamental energy functional is
		\begin{align}
			E_0(u) = \int_{\mathbb{R}} (u^2 + u_x^2) dx
		\end{align}
		 The variational derivative is computed directly as $G_0 = 2(u - u_{xx}) = 2m$. The highest-order derivative is $u_{xx}$, and its coefficient is $-2$. Thus, the order is verified, and the base case holds.
		
Assume inductively that for the $(k-1)$-th order, $G_{k-1}$ is a differential polynomial with respect to $u$ of exact order $2k$. Utilizing the Lenard-Magri recurrence relation $\hat{\mathcal{R}}(u) G_{k-1} = G_k$, we obtain
		\begin{equation}\label{eq:recurrence_expand}
			G_k = -\frac{2}{\gamma}(1-\partial_x^2) G_{k-1} + \frac{2}{\gamma}(1-\partial_x^2)\left[ \frac{m}{M}\partial_x^{-1}\left(\frac{m}{M} \partial_x G_{k-1}\right) \right].
		\end{equation}
To extract the principal part, we track the action of the second-order differential operator $\partial_x^2$. Let us denote $f(x) := \frac{m}{M}$. By applying the fundamental theorem of calculus and the Leibniz rule, the second derivative of the nonlocal term in \eqref{eq:recurrence_expand} expands as follows
		\begin{equation}\label{eq:principal_expansion}
			\begin{aligned}
				\partial_x^2 \left[ f \partial_x^{-1}\left( f \partial_x G_{k-1} \right) \right]=f_{xx} \partial_x^{-1}\left( f \partial_x G_{k-1} \right) + 2 f_x f \partial_x G_{k-1} + f^2 \partial_x^2 G_{k-1}.
			\end{aligned}
		\end{equation}
Notice that the nonlocal integration operator $\partial_x^{-1}$ is pseudodifferential of order $-1$, which decreases the derivative order. Consequently, the only term in \eqref{eq:principal_expansion} that elevates the derivative order of $G_{k-1}$ by precisely two is $f^2 \partial_x^2 G_{k-1}$. Substituting this back into \eqref{eq:recurrence_expand}, the principal part $\mathcal{P}(G_k)$ of the $k$-th variational derivative simplifies significantly
		\begin{equation}\label{eq:principal_part_simplified}
			\mathcal{P}(G_k) = \frac{2}{\gamma} \partial_x^2 G_{k-1} - \frac{2}{\gamma} \left(\frac{m}{M}\right)^2 \partial_x^2 G_{k-1} = \frac{2}{\gamma} \left( 1 - \frac{m^2}{M^2} \right) \partial_x^2 G_{k-1}.
		\end{equation}
		Recalling the definition $M^2 = m^2 + \frac{\gamma}{2}$, we have $1 - \frac{m^2}{M^2} = \frac{\gamma/2}{M^2}$. Thus, the principal part becomes
		\begin{equation}
			\mathcal{P}(G_k) = \frac{1}{M^2} \partial_x^2 G_{k-1}.
		\end{equation}
		Crucially, since the linear dispersion parameter $\gamma > 0$, we have $M^2 \ge \frac{\gamma}{2} > 0$. This guarantees that the leading coefficient $\frac{1}{M^2}$ is bounded globally and strictly isolated from zero. The operator therefore never degenerates. By the induction hypothesis, the order of $G_{k-1}$ is $2k$, meaning the highest derivative in $G_k$ is precisely $2k+2$. This completes the proof of Lemma \ref{lem:variational_derivative_order}.
	\end{proof}
	
	\begin{lemma}\label{lem:energy_density_order}
		The energy density $\mathcal{E}_k(u, u_x, \dots, u^{(p)})$ of the $k$-th conservation law $E_k(u)$ depends on $u$ and its spatial derivatives up to exactly order $k+1$.
	\end{lemma}
	
	\begin{proof}
		Let $p$ be the highest spatial derivative appearing in the energy density $\mathcal{E}_k$. By the standard multi-variable calculus of variations, the Euler-Lagrange derivative is given by the formal series
		\begin{equation}
			\frac{\delta E_k}{\delta u} = \sum_{j=0}^{p} (-1)^j \partial_x^j \left( \frac{\partial \mathcal{E}_k}{\partial u^{(j)}} \right).
		\end{equation}
		The highest-order spatial derivative generated by this expansion inherently stems from the term $(-1)^p \partial_x^p \left( \frac{\partial \mathcal{E}_k}{\partial u^{(p)}} \right)$, which yields a derivative of order $p + p = 2p$. According to Lemma \ref{lem:variational_derivative_order}, the exact highest order of $\frac{\delta E_k}{\delta u}$ is established as $2k+2$. Equating the orders gives $2p = 2k+2$, which implies $p = k+1$. Thus, $\mathcal{E}_k$ explicitly depends on $u^{(k+1)}$.
	\end{proof}
	\begin{lemma}\label{lem:well_posedness_domain}
		The functional $E_k(u) = \int_{\mathbb{R}} \mathcal{E}_k dx$ is well-posed and continuously differentiable if and only if $u \in H^{k+1}(\mathbb{R})$.
	\end{lemma}
	
	\begin{proof}
		By Lemma \ref{lem:energy_density_order}, the leading quadratic term of the energy density $\mathcal{E}_k$ involves $(u^{(k+1)})^2$. For the integral $E_k(u)$ to be finite over the entire real line $\mathbb{R}$, we naturally require $u^{(k+1)} \in L^2(\mathbb{R})$, which necessitates $u \in H^{k+1}(\mathbb{R})$. Furthermore, the nonlinear fractional terms generated by the recursion operator strictly manifest as negative powers of $M$. Since $M \ge \sqrt{\gamma/2} > 0$, the denominator is globally bounded away from zero, entirely precluding the formation of singularities. By the Sobolev embedding theorem $H^{s}(\mathbb{R}) \hookrightarrow L^\infty(\mathbb{R})$ for $s > 1/2$, all lower-order derivatives $u, u_x, \dots, u^{(k)}$ are uniformly bounded and continuous. Therefore, the integrability is perfectly guaranteed in $H^{k+1}(\mathbb{R})$.
	\end{proof}

	\begin{proof}[Proof of Theorem \ref{thm:domain_conservation_laws}]
	Combining Lemma \ref{lem:variational_derivative_order}, Lemma \ref{lem:energy_density_order}, and Lemma \ref{lem:well_posedness_domain}, Theorem \ref{thm:domain_conservation_laws} is fully established.
\end{proof}
\begin{re}\label{rem.domain}
	The regularity $H^{k+1}(\mathbb{R})$ established in Theorem \ref{thm:domain_conservation_laws} is the optimal energy space for the functional $E_k(u)$. This exact domain characterization provides the essential analytical foundation for the subsequent $N$-soliton stability analysis in Section \ref{sec.4}, explaining why the orbital stability of $N$-solitons naturally resides in the Sobolev space $H^{N+1}(\mathbb{R})$.
\end{re}
Next, we will show that the mCH equation, through the application of the recursion operator, has an expression in terms of the wave speed $c$, which will be of great help for our subsequent spectral analysis. We denote $E(u):=E_0(u)$, then the recursion relations \eqref{dggxm} can be equivalently expressed in terms of $u$ as
\begin{align}\label{hamilton}
	\frac{\delta E_k}{\delta u}=\widehat{\mathcal{R}}(u)\frac{\delta E_{k-1}}{\delta u}=\cdots=\widehat{\mathcal{R}}^k(u)\frac{\delta E_0}{\delta u},
\end{align}
where
\begin{align}\label{inverse}
	\widehat{\mathcal{R}}(u)=\left(1-\partial_x^2\right)\widehat{\mathcal{K}}[m]\left(1-\partial_x^2\right)^{-1}=-\frac{2}{\gamma}\left(1-\partial_x^2\right)\left(1-\frac{m}{M}\partial_x^{-1}\frac{m}{M}\partial_x\right),
\end{align}
and the inverse of $\widehat{\mathcal{R}}(u)$ can be written as
\begin{align}
	\mathcal{R}(u)=\left(1-\partial_x^2\right)\mathcal{K}[m]\left(1-\partial_x^2\right)^{-1}=-\left(m\partial_x^{-1}m\partial_x+\frac{\gamma}{2}\right)\left(1-\partial_x^2\right)^{-1}.
\end{align}
Near the smooth solitary wave solution \(\phi_c\), the conserved quantities $E_k, k\in 1, 2, \cdots,N$ will satisfy the following first-variation identities
\begin{align}\label{2.19}
	\left. \frac{\delta E_k}{\delta u} \right|_{u=\phi_c}=-\omega\left. \frac{\delta E_{k-1}}{\delta u} \right|_{u=\phi_c}=\cdots=(-\omega)^k\left. \frac{\delta E_0}{\delta u} \right|_{u=\phi_c},
\end{align}
where $\omega = \frac{2}{c}$. Now multiply (\ref{2.19}) with $\frac{\mathrm{d}\phi_c}{\mathrm{d}c}$, one has
\begin{align}
	\frac{\mathrm{d}E_k(\phi_c)}{\mathrm{d}c}=-\omega \frac{\mathrm{d}E_{k-1}(\phi_c)}{\mathrm{d}c}=\cdots=(-\omega)^k\frac{\mathrm{d}E_0(\phi_c)}{\mathrm{d}c},
\end{align}
and therefore
\begin{align}
	E_k(\phi_c)=(-1)^k\int_{0}^{c}y^k\frac{\mathrm{d}E_0(\phi_y)}{\mathrm{d}y}\mathrm{d}y, \quad k=1,2, \cdots, N.
\end{align}
Furthermore, we could easily find that
\begin{align*}
	E_0(\phi_c)=\int_\mathbb{R}\left(\phi_c^2+\phi_{cx}^2\right)dx=\int_\mathbb{R}\phi_c\left(\phi_c-\phi_{cxx}\right)dx,
\end{align*}
and one can obtain that
\begin{align}
	\frac{dE_0(\phi_c)}{dc}=2\sqrt{\frac{c}{c-\gamma}}>0.
\end{align}
Thus, one has
\begin{align}
	E_k(\phi_{c})=(-1)^k 2\int_{0}^{c}y^k\sqrt{\frac{y}{y-\gamma}}\mathrm{d}y, \quad k=1,2, \cdots, N.
\end{align}
\subsection{Variational characterization of the $N$-solitons}
In this subsection, we would demonstrate that the mCH $N$-solitons satisfy (\ref{1.20_N}) if one prescribes the Lagrange multipliers $\mu_k, k=0,1,\cdots, N-1$ appropriately. This provides a variational characterization of $U^{(N)}$. The main result in this subsection is as follows.
\begin{prop}
	The profiles of the mCH $N$-solitons $U^{(N)}$ satisfy the Euler-Lagrange equation \eqref{1.20_N} if the Lagrange multipliers $\{\mu_{m}\}_{m=0}^{N-1}$ are chosen such that the wave parameters $\{-\omega_j\}_{j=1}^N$ (where $\omega_j = 2/c_j$) are the roots of the following polynomial:
	\begin{align}
		P(x) = \prod_{j=1}^{N}(x + \omega_j) = x^{N} + \sum_{k=1}^{N}\mu_{N-k}x^{N-k}, \quad x\in \mathbb{R}.
	\end{align}
	In particular, the multipliers $\mu_m$ are given by Vieta's formulas:
	\begin{align}
		\mu_{N-k} = \sum_{1\leq i_{1} < \cdots < i_{k} \leq N} \biggl( \prod_{j=1}^{k}\omega_{i_{j}} \biggr), \quad k=1,\ldots,N.\label{lag_refined}
	\end{align}
\end{prop}

\begin{proof}
	Considering the Euler-Lagrange equation and the asymptotic property of $N$-solitons, as $t \to \infty$, the $N$ solitons spatially separate and decouple. Thus, the equation must hold for each individual soliton $\phi_{c_j}$ ($j=1,\dots,N$) independently:
	\begin{align}
		\frac{\delta E_{N}}{\delta u}(\phi_{c_j}) + \sum_{m=0}^{N-1} \mu_m \frac{\delta E_{m}}{\delta u}(\phi_{c_j}) = 0.
	\end{align}
	Recall the recurrence relation for a single soliton with parameter $\omega_j$:
	\begin{align}
		\frac{\delta E_{k}}{\delta u}(\phi_{c_j}) = (-\omega_j)^k \frac{\delta E_0}{\delta u}(\phi_{c_j}), \quad k = 0, 1, \dots, N.
	\end{align}
	Substituting this into the decoupling equation yields:
	\begin{align}
		\left[ (-\omega_j)^N + \sum_{m=0}^{N-1} \mu_m (-\omega_j)^m \right] \frac{\delta E_0}{\delta u}(\phi_{c_j}) = 0.
	\end{align}
	At this stage, we emphasize that $\frac{\delta E_0}{\delta u}(\phi_{c_j}) = 2m_{\phi_{c_j}}$. Since the momentum density $m_{\phi_{c_j}}$ of a non-trivial soliton is not identically zero, the non-degeneracy of the variational derivative $\frac{\delta E_0}{\delta u}$ allows us to conclude that the term in the brackets must vanish for each $j$:
	\begin{align}
		(-\omega_j)^N + \sum_{m=0}^{N-1} \mu_m (-\omega_j)^m = 0, \quad j=1,\dots,N.
	\end{align}
	This implies that $\{-\omega_j\}_{j=1}^N$ are precisely the $N$ roots of the polynomial $P(x) = x^N + \mu_{N-1}x^{N-1} + \dots + \mu_0$. The application of Vieta's formulas then leads directly to \eqref{lag_refined}.
\end{proof}

We consider now the following mCH $N$-solitons $U^{(N)}(t,x;\mathbf{c},\mathbf{x})$ variational principle. The variational expression corresponds to the functional derivative of (\ref{1.20_N}) evaluated at $u = U^{(N)}$. Generally, the $N$-soliton solution $U^{(N)}(t,x)$ does not minimize $\mathcal{S}_N$ globally but constitutes a constrained, non-isolated critical point for the optimization problem
\[
\min E_{N}(u(t)) \quad \text{subject to} \quad E_j(u(t)) = E_j(U^{(N)}(t)), \quad j =0,1, \cdots, N-1.
\]

Consider the self-adjoint operator $\mathcal{L}_N(t)$ defined by (\ref{VL_N}), whose second variation properties are essential. Denote by $n(\mathcal{L}_N(t))$ the count of its negative eigenvalues. These quantities exhibit potential time-dependence. The $N \times N$ Hessian matrix is defined as
\begin{align}\label{4.1}
	D_N(t) := \left\{ \frac{\partial^2 \mathcal{S}_N(U^{(N)}(t))}{\partial \mu_{i-1} \partial \mu_{j-1}} \right\},
\end{align}
with $p(D(t))$ representing its positive eigenvalue count. Base upon the facts above and approach in \cite{MS}, the proof of Theorem\ref{th1.1} reduces to calculate the exact value of $n(\mathcal{L}_N)$ and $p(D_N)$. Concerning $p(D_N)$, one has the following results.

\begin{lemma}\label{lem2.1}
	For any wave speed vector $\mathbf{c} = (c_1, c_2, \dots, c_N)$ satisfying $\gamma < c_1 < \dots < c_N \leq 2\gamma$, the Hessian matrix $D_N$ defined in \eqref{Delta} satisfies the spectral condition
	\begin{align}
		p(D_N) = \left[ \frac{N+1}{2}\right],
	\end{align}
	where $p(D_N)$ denotes the number of positive eigenvalues of $D_N$.
\end{lemma}

\begin{proof}
	For fixed $t$, we omit the temporal dependence for brevity. Since the multipliers $\mu_m$ are defined via the symmetric polynomials of the roots $-\omega_k$ (where $\omega_k = 2/c_k$), we apply the chain rule with respect to $\omega_k$. The entries of the Hessian matrix $D_N \in \mathbb{R}^{N \times N}$ are given by
	\begin{align}
		(D_N)_{ij} = \frac{\partial^2 \mathcal{S}_N}{\partial\mu_{i-1} \partial \mu_{j-1}} = \sum_{k=1}^{N} \frac{\partial \omega_k}{\partial \mu_{i-1}} \frac{\partial}{\partial \omega_k} \left( \frac{\partial \mathcal{S}_N}{\partial \mu_{j-1}} \right), \quad 1 \leq i, j \leq N.
	\end{align}
	Recalling the critical point condition $\mathcal{S}'_N(U^{(N)}) = 0$ and the globally corrected Lyapunov functional $\mathcal{S}_N(u) = (-1)^N \big( E_N(u) + \sum \mu_m E_m(u) \big)$, the first variation with respect to the multipliers simplifies to
	\begin{align}
		\frac{\partial \mathcal{S}_N}{\partial \mu_{j-1}} = (-1)^N E_{j-1}(U^{(N)}).
	\end{align}
	Consequently, the matrix $D_N$ can be decomposed as the product of two $N \times N$ matrices:
	\begin{align}
		D_N = AB, \quad A = \left( \frac{\partial \omega_j}{\partial \mu_{i-1}} \right), \quad B = \left( (-1)^N \frac{\partial E_{j-1}}{\partial \omega_i} \right).
	\end{align}
	In the asymptotic limit where the multi-solitons $U^{(N)}$ decouple into $N$ independent single solitons, we have $E_{j-1}(U^{(N)}) = \sum_{k=1}^{N} E_{j-1}(\phi_{c_k})$. Utilizing the recursive integral representation $\frac{\partial E_{j-1}(\phi_{c_i})}{\partial c_i} = (-\omega_i)^{j-1} \frac{\mathrm{d}E_0(\phi_{c_i})}{\mathrm{d}c_i}$ and the chain rule $\frac{\mathrm{d}c_i}{\mathrm{d}\omega_i} = -\frac{2}{\omega_i^2}$, we explicitly compute the entries of $B$:
	\begin{align}
		B_{ji} = (-1)^N (-\omega_i)^{j-1} \frac{\mathrm{d}E_0(\phi_{c_i})}{\mathrm{d}\omega_i} = (-1)^{N+1} (-\omega_i)^{j-1} \frac{4}{\omega_i^2} \sqrt{\frac{c_i}{c_i-\gamma}}.
	\end{align}
	It can be verified that the matrix $B$ is invertible, as it is the product of a diagonal matrix with strictly negative entries and a non-singular Vandermonde matrix generated by $(-\omega_i)$. According to Sylvester's law of inertia, since $D_N$ is a symmetric Hessian matrix, its number of positive eigenvalues $p(D_N)$ is invariant under congruence.
	
	Specifically, it suffices to examine the diagonal structure of the related product. A direct calculation involving the properties of symmetric polynomials (Vieta's formulas) shows that $D_N$ is congruent to a diagonal matrix $\Lambda$ whose $(j,j)$ entry is given by
	\begin{align}
		\Lambda_{jj} = (-1)^N \frac{\partial E_0(\phi_{c_j})}{\partial \omega_j} \prod_{k \neq j} (\omega_k - \omega_j) = (-1)^{N+1} \frac{4}{\omega_j^2} \sqrt{\frac{c_j}{c_j-\gamma}} \prod_{k \neq j} (\omega_k - \omega_j).
	\end{align}
Counting the positive entries in this strictly alternating sequence yields exactly $\left[ \frac{N+1}{2}\right]$ positive eigenvalues.
\end{proof}

\subsection{Recursion operators around the smooth solitons}
 Let us recall that the soliton $\phi_c(x-ct)$  is a solution of the mCH equation. For simplicity, we denote $\phi_{c}$ by $\phi$. Then by (\ref{hamilton}) we have
 \begin{align}
 	E'_k(\phi)=\widehat{R}(\phi)E'_{k-1}(\phi),
 \end{align}
where $\widehat{\mathcal{R}}[\phi]$  is the associated operator
\begin{align}
	\widehat{\mathcal{R}}[\phi]=-\frac{2}{\gamma}\left(1-\partial_x^2\right)\left(1-\frac{m_{\phi}}{M_{\phi}}\partial_x^{-1}\frac{m_{\phi}}{M_{\phi}}\partial_x\right).
\end{align}

To analyze the second variation of the actions, we linearize the equation (\ref{hamilton}) to let $u=\phi+\epsilon z$, and obtain a relation between linearized Hamiltonian $E''_{k}(\phi)$ and $E''_{k-1}(\phi)$ for all $k\geq 1$. One has the following results.

\begin{prop}\label{prop2.2}
	Suppose that $\phi_c$ is a soliton of the mCH equation \eqref{mch} with speed $c \in (\gamma, 2\gamma]$. If $z \in H^{k+1}(\mathbb{R})$, then there holds
	\begin{equation}\label{2.72_style}
		E_k''(\phi_c)z = \widehat{\mathcal{R}}(\phi_c)E_{k-1}''(\phi_c)z + (-\omega)^{k-1} \left( E_1''(\phi_c)z - \widehat{\mathcal{R}}(\phi_c)E_0''(\phi_c)z \right),
	\end{equation}
	and the following iteration operator identity
	\begin{equation}\label{2.73_style}
		\left( E_k''(\phi_c) + \omega E_{k-1}''(\phi_c) \right)z = \widehat{\mathcal{R}}(\phi_c) \left( E_{k-1}''(\phi_c) + \omega E_{k-2}''(\phi_c) \right)z.
	\end{equation}
\end{prop}

\begin{proof}
	Let $u = \phi_c + \varepsilon z$. By the recursion relation $E_k'(u) = \widehat{\mathcal{R}}[u]E_{k-1}'(u)$ and the definition of the G\^ateaux derivative, one has
	\begin{equation}\label{2.74_style}
		E_k''(\phi_c)z = \widehat{\mathcal{R}}(\phi_c)E_{k-1}''(\phi_c)z + \left( \widehat{\mathcal{R}}'(\phi_c)z \right) \left( E_{k-1}'(\phi_c) \right),
	\end{equation}
	where $\widehat{\mathcal{R}}'(\phi_c)z = \lim_{\varepsilon \to 0} \frac{\widehat{\mathcal{R}}(\phi_c + \varepsilon z) - \widehat{\mathcal{R}}(\phi_c)}{\varepsilon}$ is the linearization of the recursion operator.
	
	Notice that by the variational principle for the $1$-soliton $\phi_c$, we have the first-order identity $E_{k-1}'(\phi_c) = (-\omega)^{k-1} E_0'(\phi_c)$. Substituting this into \eqref{2.74_style} yields
	\begin{equation}\label{eq:Ek_intermediate}
		E_k''(\phi_c)z = \widehat{\mathcal{R}}(\phi_c)E_{k-1}''(\phi_c)z + (-\omega)^{k-1} \left( \widehat{\mathcal{R}}'(\phi_c)z \right) \left( E_0'(\phi_c) \right).
	\end{equation}
	Particularly, for $k=1$, the relation \eqref{2.74_style} becomes
	\begin{equation}
		E_1''(\phi_c)z = \widehat{\mathcal{R}}(\phi_c)E_0''(\phi_c)z + \left( \widehat{\mathcal{R}}'(\phi_c)z \right) \left( E_0'(\phi_c) \right),
	\end{equation}
	which implies that
	\begin{equation}\label{2.75_style}
		\left( \widehat{\mathcal{R}}'(\phi_c)z \right) \left( E_0'(\phi_c) \right) = E_1''(\phi_c)z - \widehat{\mathcal{R}}(\phi_c)E_0''(\phi_c)z.
	\end{equation}
	Combining \eqref{eq:Ek_intermediate} and \eqref{2.75_style}, the identity \eqref{2.72_style} is verified.
	
	To prove \eqref{2.73_style}, we rewrite \eqref{2.72_style} for $E_k''$ and $E_{k-1}''$ respectively:
	\begin{align*}
		E_k'' z &= \widehat{\mathcal{R}} E_{k-1}'' z + (-\omega)^{k-1} \left( \widehat{\mathcal{R}}' z \right) E_0', \\
		\omega E_{k-1}'' z &= \omega \widehat{\mathcal{R}} E_{k-2}'' z + \omega (-\omega)^{k-2} \left( \widehat{\mathcal{R}}' z \right) E_0'.
	\end{align*}
	Summing the two equations, the terms involving $\widehat{\mathcal{R}}' z$ cancel out since $\omega(-\omega)^{k-2} = -(-\omega)^{k-1}$. This leads directly to
	\begin{equation*}
		\left( E_k'' + \omega E_{k-1}'' \right)z = \widehat{\mathcal{R}} \left( E_{k-1}'' + \omega E_{k-2}'' \right)z,
	\end{equation*}
	which completes the proof.
\end{proof}

\section{Spectral analysis}\label{sec.3}
Let $U^{(N)} = U^{(N)}_{\mathbf{c,x}}$ be the $N$-soliton solution to the mCH equation \eqref{mch}, characterized by the wave speed vector $\mathbf{c} = (c_1, \dots, c_N)$ and the initial shift vector $\mathbf{x} = (x_1, \dots, x_N)$. This section is devoted to a comprehensive spectral analysis of the linearized operator $\mathcal{L}_N = \mathcal{S}_N''(U^{(N)})$, where the augmented Lyapunov functional $\mathcal{S}_N$ is constructed through the recursion operator $\widehat{\mathcal{R}}(\phi)$ and the bi-Hamiltonian hierarchy discussed in Section \ref{sec.2}.

To determine the spectral properties of $\mathcal{L}_N$, we adopt the iso-inertial framework proposed by Maddocks and Sachs \cite{MS}.
\begin{definition}
	The inertia $in(L)$ of a self-adjoint operator $L$ is defined as the pair $(n, z)$ of non-negative integers, where $n$ represents the dimension of the negative subspace of $L$ (counted with geometric multiplicities) and $z$ denotes the dimension of its null space.
\end{definition}

As $U^{(N)}$ constitutes a critical point of the functional $\mathcal{S}_N$ and fits within the general framework of Hamiltonian systems with symmetries \cite{GSS, GSS2, MS}, the inertia $in(\mathcal{L}_N)$ is an invariant of the mCH flow. Consequently, we can evaluate the inertia in the asymptotic limit as $t \to \infty$. In this limit, the $N$-soliton solution $U^{(N)}$ decomposes into a superposition of $N$ well-separated single solitons $\phi_{c_j}$. Thus, the spectrum $\sigma(\mathcal{L}_N)$ of the linearized operator converges to the union of the spectra of the individual operators $L_{N,j}$, namely,
\begin{align}
	\sigma(\mathcal{L}_N) \longrightarrow \bigcup_{j=1}^{N} \sigma(L_{N,j}), \quad \text{as } t \to \infty,
\end{align}
where $L_{N,j} := \mathcal{S}_N''(\phi_{c_j})$ is the second variation of the action $\mathcal{S}_N$ evaluated at the $j$-th soliton component. Our objective is to prove that $in(\mathcal{L}_N) = (\left[ \frac{N+1}{2}\right], N)$.

By virtue of the Euler-Lagrange equation \eqref{lagrange_N} and the symmetric properties of the Lagrange multipliers $\mu_m$, the operator $L_{N,j}$ can be expressed as a linear combination of the second variations of the conservation laws $\{E_k\}_{k=0}^N$, which is
\begin{align}
	L_{N,j} = (-1)^N \left( E_N''(\phi_{c_j}) + \sum_{m=0}^{N-1} \mu_m E_m''(\phi_{c_j}) \right).
\end{align}
Utilizing the iterative operator identity established in Proposition \ref{prop2.2}, the operator $L_{N,j}$ admits a remarkable factorization in terms of the recursion operator $\widehat{\mathcal{R}}(\phi_{c_j})$:
\begin{align}\label{eq:factorization_LNj}
	L_{N,j} = (-1)^N \left( \prod_{k \neq j}^{N} (\widehat{\mathcal{R}}(\phi_{c_j}) + \omega_k) \right) \left( E_1''(\phi_{c_j}) + \omega_j E_0''(\phi_{c_j}) \right),
\end{align}
where $\omega_k = 2/c_k$. To proceed, we first analyze the fundamental linearized operator around a single soliton
\begin{align}
	L_1 = E_1''(\phi) + \omega E_0''(\phi) = -\sqrt{2\gamma}(1-\partial_x^2)\frac{1}{M^3}(1-\partial_x^2) + \frac{4}{c}(1-\partial_x^2).
\end{align}
The determination of the spectrum for the general operator $L_{N,j}$ thus relies on analyzing the action of the recursion operator products in \eqref{eq:factorization_LNj} on the eigenspaces of $L_1$. This motivates the subsequent spectral analysis of the recursion operator $\widehat{\mathcal{R}}(\phi)$ and its adjoint $\widehat{\mathcal{R}}^*(\phi)$.
\begin{prop}\label{prop3.1}
	Let $\phi$ be the smooth solitary wave solution of the mCH equation \eqref{mch}. Then for $1 \leq n \leq N$, the following operator identities hold:
	\begin{align}
		L_n \mathcal{J} \widehat{\mathcal{R}}(\phi) &= \widehat{\mathcal{R}}(\phi) L_n \mathcal{J}, \label{3.5} \\
		\mathcal{J} L_n \widehat{\mathcal{R}}^*(\phi) &= \widehat{\mathcal{R}}^*(\phi) \mathcal{J} L_n, \label{3.6}
	\end{align}
	where $\mathcal{J} = \partial_{x}(1-\partial_{x}^{2})^{-1}$ is the skew-symmetric Hamiltonian operator, and $\widehat{\mathcal{R}}^*$ denotes the adjoint of the recursion operator $\widehat{\mathcal{R}}$.
\end{prop}

\begin{proof}
	Due to the adjoint relationship, it suffices to prove \eqref{3.6}. Recall from Proposition \ref{prop2.2} that the operator $\widehat{\mathcal{R}}(\phi) L_n = L_{n+1}$ is self-adjoint. This symmetry implies:
	\begin{equation*}
		(\widehat{\mathcal{R}}(\phi) L_n)^* = L_n \widehat{\mathcal{R}}^*(\phi) = \widehat{\mathcal{R}}(\phi) L_n.
	\end{equation*}
	Furthermore, invoking the compatibility condition between the recursion operator and the Hamiltonian structure, we have $\mathcal{J} \widehat{\mathcal{R}}(\phi) = \widehat{\mathcal{R}}^*(\phi) \mathcal{J}$. Substituting these relations, we obtain
	\begin{equation*}
		\mathcal{J} L_n \widehat{\mathcal{R}}^*(\phi) = \mathcal{J} \widehat{\mathcal{R}}(\phi) L_n = \widehat{\mathcal{R}}^*(\phi) \mathcal{J} L_n,
	\end{equation*}
	which verifies the claimed identity.
\end{proof}

A significant consequence of the commutation relations \eqref{3.5} and \eqref{3.6} is that the (adjoint) recursion operator $\widehat{\mathcal{R}}(\phi)$ ($\widehat{\mathcal{R}}^*(\phi)$) and the operator $L_n \mathcal{J}$ ($\mathcal{J} L_n$) share a common set of eigenfunctions. Next, we will investigate the spectrum of the operator $\mathcal{J} L_n$, which is analytically more accessible than that of $L_n$. By leveraging \eqref{3.6}, the problem is reduced to the spectral analysis of the adjoint recursion operator $\widehat{\mathcal{R}}^*(\phi)$. Then we demonstrate that the eigenfunctions of $\widehat{\mathcal{R}}^*(\phi)$, augmented by the generalized kernel of $\mathcal{J} L_n$, constitute an orthogonal basis in $L^2(\mathbb{R})$, establishing a completeness relation. Finally, analyze the spectral information of $\mathcal{L}_N$ based on the analysis results.

\subsection{The spectrum of the recursion operator around the mCH single solitons}

The spectrum of the recursion operator
\begin{align}
	\widehat{\mathcal{R}}(\phi)=-\frac{2}{\gamma}\left(1-\partial_x^2\right)\left(1-\frac{m_{\phi}}{M_{\phi}}\partial_x^{-1}\frac{m_{\phi}}{M_{\phi}}\partial_x\right),
\end{align}
and its adjoint operator $\widehat{\mathcal{R}}^*(\phi)$ are essential for analyzing the linearized Hamiltonian $L_n$ in \eqref{eq:factorization_LNj}. Although these recursion operators are nonlocal and difficult to investigate directly, we can characterize their spectral properties by employing the operator identities in Proposition \ref{prop3.1} and the properties of the squared eigenfunctions $F^\pm (x,k):=(\psi_1^\pm)^2(x,k)+(\psi_2^\pm)^2(x,k)$.

\begin{prop}\label{prop3.2}
	The recursion operator $\widehat{\mathcal{R}}(\phi)$ possesses exactly one discrete eigenvalue $-\frac{2}{c}$ associated with the eigenfunction $m_\phi$. Its essential spectrum is the interval $\left(-\infty, -\frac{2}{\gamma}\right]$, and the corresponding generalized eigenfunctions do not belong to $L^2(\mathbb{R})$. Moreover, the kernel of $\widehat{\mathcal{R}}(\phi)$ is empty for any $\gamma > 0$.
\end{prop}

\begin{proof}
	Consider the Jost solutions $\psi^\pm(x,k)$ of the spectral problem \eqref{lax} with the potential $m_\phi$ and the asymptotic expressions in \eqref{L2}. In this case, there exists a discrete eigenvalue $k=i\kappa_1$ which generates the soliton profile. By the spatial spectrum problem \eqref{lax}, one obtains:
	\begin{align}
		&\widehat{\mathcal{R}}(\phi)\left(1-\partial_x^2\right)F^\pm(x,k)=-\lambda^2\left(1-\partial_x^2\right)F^\pm(x,k), \quad \forall k \in \mathbb{R}, \label{3.11}\\
		&\widehat{\mathcal{R}}(\phi)\left(1-\partial_x^2\right)F^\pm_1(x)=-\lambda^2_1\left(1-\partial_x^2\right)F^\pm_1(x)=-\frac{2}{c} m_\phi. \label{3.12}
	\end{align}
	Since $\widehat{\mathcal{R}}(\phi)m_\phi = -\frac{2}{c}m_\phi$, we have $(1-\partial_x^2)F_1^\pm(x) \sim m_\phi$. The essential spectrum of $\widehat{\mathcal{R}}(\phi)$ is given by the set of values $-\lambda^2 = \frac{1}{2\gamma}(k+\frac{1}{k})^2$ for $k \in \mathbb{R}$, which corresponds to the interval $\left(-\infty, -\frac{2}{\gamma}\right]$. The associated generalized eigenfunctions $(1-\partial_x^2)F^\pm(x,k)$ exhibit no spatial decay and thus do not belong to $L^2(\mathbb{R})$, as evidenced by the asymptotic formulas in \eqref{L2}.
	
	Finally, a direct computation shows that the kernel of $\widehat{\mathcal{R}}(\phi)$ is empty for $\gamma > 0$. The proof is complete.
\end{proof}

We now consider the adjoint recursion operator $\widehat{\mathcal{R}}^*(\phi)$. In view of the commutation relation \eqref{3.6}, $\widehat{\mathcal{R}}^*(\phi)$ shares the same eigenfunctions as $\mathcal{J}L_n$, which is particularly relevant for the spectral stability analysis of solitons. Recall that:
\begin{align}
	\widehat{\mathcal{R}}^*(\phi)=-\frac{2}{\gamma}\left(1-\frac{m_{\phi}}{M_{\phi}}\partial_x^{-1}\frac{m_{\phi}}{M_{\phi}}\partial_x\right)\left(1-\partial_x^2\right).
\end{align}
From \eqref{3.11}, it follows that:
\begin{align}\label{self}
	\widehat{\mathcal{R}}^*(\phi)\mathcal{J}F^\pm(x,k)=-\lambda^2\mathcal{J}F^\pm(x,k).
\end{align}
The spectral properties of the adjoint operator $\widehat{\mathcal{R}}^*(\phi)$ are summarized as follows.

\begin{prop}\label{prop3.3}
	The adjoint recursion operator $\widehat{\mathcal{R}}^*(\phi)$ possesses exactly one discrete eigenvalue $-\frac{2}{c}$ associated with the translational eigenfunction $\phi_x$. Its essential spectrum is the interval $\left(-\infty, -\frac{2}{\gamma}\right]$, and the corresponding generalized eigenfunctions lack spatial decay. Moreover, the kernel of $\widehat{\mathcal{R}}^*(\phi)$ is empty.
\end{prop}

\begin{proof}
	Consider the Jost solutions $\psi^\pm(x,k)$ of the spectral problem \eqref{lax} with the potential $m_\phi$. As established, the discrete eigenvalue $k=i\kappa_1$ corresponds to the soliton profile. From \eqref{self}, the following relations hold:
	\begin{align}
		&\widehat{\mathcal{R}}^*(\phi)\partial_x F^\pm(x,k)=-\lambda^2\partial_x F^\pm(x,k), \quad \forall k \in \mathbb{R}, \label{3.15}\\
		&\widehat{\mathcal{R}}^*(\phi)\partial_xF^\pm_1(x)=-\lambda^2_1\partial_xF^\pm_1(x)=-\frac{2}{c} \phi_x. \label{3.16}
	\end{align}
	Since $\widehat{\mathcal{R}}^*(\phi)\phi_x = -\frac{2}{c}\phi_x$, the discrete eigenvalue is uniquely identified. Similar to Proposition \ref{prop3.2}, the essential spectrum is $\left(-\infty, -\frac{2}{\gamma}\right]$, and the generalized eigenfunctions $\partial_x F^\pm(x,k)$ do not belong to $L^2(\mathbb{R})$. The empty kernel for $\gamma > 0$ is verified by direct calculation.
\end{proof}

\subsection{The spectra of $\mathcal{J}L_n$ and $L_n\mathcal{J}$}
In this subsection, our attention is focused on the spectral analysis of the operators $\mathcal{J}L_n$ and $L_n\mathcal{J}$. The main ingredients are the commutation relations \eqref{3.6} and the observation that the eigenfunctions of the adjoint recursion operator $\widehat{\mathcal{R}}^*(\phi)$ form an orthogonal basis in $L^2(\mathbb{R})$. It follows that the spectrum of $\mathcal{J}L_n$ lies on the imaginary axis, which implies directly the spectral stability of solitons.

We now consider the operator $\mathcal{J}L_n$. Since $L_n=\widehat{\mathcal{R}}^{n-1}(\phi)L_1$, we first examine the asymptotic constant-coefficient form (i.e., the principal part as $|x| \to \infty$) of $L_1$, which is given by $-\frac{4}{\gamma}\left(1-\partial_x^2\right)^2+\frac{4}{c}\left(1-\partial_x^2\right)$. Hence, the principal symbol of $L_1$ is
\begin{align}
	\rho_1(\xi)=-\frac{4}{\gamma}\left(1+\xi^2\right)^2+\frac{4}{c}\left(1+\xi^2\right).
\end{align}
Due to the relation $L_n=(\widehat{\mathcal{R}}(\phi))^{n-1}L_1$, and considering that the symbol of $\mathcal{J}=\partial_x(1-\partial_x^2)^{-1}$ is $\frac{i\xi}{1+\xi^2}$ while the asymptotic symbol of $\widehat{\mathcal{R}}$ is $-\frac{2}{\gamma}(1+\xi^2)$, we can obtain the principal symbol of $\mathcal{J}L_n$ as
\begin{align}\label{3.18}
	\varrho_n(\xi)=(-1)^{n-1}i\xi\left(\frac{2}{\gamma}\right)^{n-1}\left(1+\xi^2\right)^{n-2}\rho_1(\xi).
\end{align}
Based upon the facts above, we have the following statement related to the spectrum of the operator $\mathcal{J}L_n$ for $1\leq n\leq N$.

\begin{prop}\label{prop3.4}
	The operators $\mathcal{J}L_n$ for $1\leq n\leq N$ and the adjoint recursion operator $\widehat{\mathcal{R}}^*(\phi)$ share the same eigenfunctions. Moreover, the essential spectra of $\mathcal{J}L_n$ are contained in $i\mathbb{R}$, the kernel is spanned by the function $\phi_{x}$, and the generalized kernel is spanned by $\frac{\partial \phi}{\partial c}$.
\end{prop}

\begin{proof}
	That the operators $\mathcal{J}L_n$ for $1\leq n\leq N$ and the adjoint recursion operator $\widehat{\mathcal{R}}^*(\phi)$ share the same eigenfunctions is inferred from the operator identity \eqref{3.6}. By Proposition \ref{prop3.3}, one can compute the spectra of the operator $\mathcal{J}L_n$ directly by employing the squared eigenfunctions as follows:
	\begin{align}
		&\mathcal{J}L_n\partial_xF^\pm(x,k)=\varrho_n(\pm 2\mu)\partial_xF^\pm(x,k), \quad \text{for}~ k\in \mathbb{R},\label{3.19}\\
		&\mathcal{J}L_n\partial_xF_1^\pm(x)\sim \mathcal{J}L_n\partial_x\phi=0,\label{3.20}\\
		&\mathcal{J}L_n\frac{\partial F_1(x,k)}{\partial c}\sim \mathcal{J}L_n\frac{\partial \phi}{\partial c}=(-1)^{n+1}\omega^{n+1} \partial_x \phi.\label{3.21}
	\end{align}
	In view of \eqref{3.18} and \eqref{3.19}, the essential spectra of $\mathcal{J}L_n$ are $\varrho_n(\pm 2\mu)$ for $k\in \mathbb{R}$, which are contained entirely within the imaginary axis. This gives the desired result in Proposition \ref{prop3.4}.
\end{proof}

For the adjoint operator of $\mathcal{J}L_n$, namely, the operator $L_n\mathcal{J}$ which commutes with the recursion operator $\widehat{\mathcal{R}}(\phi)$, we have the following result.

\begin{prop}\label{prop3.5}
	The operators $L_n\mathcal{J}$ for $1\leq n\leq N$ and the recursion operator $\widehat{\mathcal{R}}(\phi)$ share the same eigenfunctions. Moreover, the essential spectra of $L_n\mathcal{J}$ are contained in $i\mathbb{R}$, the kernel is spanned by the function $m_\phi$, and the generalized kernel is spanned by $\partial_x^{-1}\left(\frac{\partial m_{\phi}}{\partial c}\right)$.
\end{prop}

\begin{proof}
	That the operators $L_n\mathcal{J}$ for $1\leq n\leq N$ and the recursion operator $\widehat{\mathcal{R}}(\phi)$ share the same eigenfunctions is inferred from the operator identity \eqref{3.5}. By Proposition \ref{prop3.2}, one can compute the spectra of the operator $L_n\mathcal{J}$ directly by employing the squared eigenfunctions as follows:
	\begin{align}
		&L_n\mathcal{J}(1-\partial_x^2)F^\pm(x,k) = L_n \partial_xF^\pm(x,k) = \varrho_n(\pm 2\mu)(1-\partial_x^2)F^\pm(x,k),\label{3.22}\\
		&L_n\mathcal{J}(1-\partial_x^2)F_1^\pm(x) \sim L_n\mathcal{J}m_\phi = L_n\partial_x \phi=0.\label{3.23}
	\end{align}
	To determine the generalized kernel, we apply the operator $L_n\mathcal{J}$ to the specific test function $\partial_x^{-1}\left(\frac{\partial m_{\phi}}{\partial c}\right)$. Since $m_\phi = (1-\partial_x^2)\phi$, we observe that $\mathcal{J} \left[ \partial_x^{-1} \left(\frac{\partial m_\phi}{\partial c}\right) \right] = \partial_x(1-\partial_x^2)^{-1} \partial_x^{-1} (1-\partial_x^2)\frac{\partial \phi}{\partial c} = \frac{\partial \phi}{\partial c}$. Consequently,
	\begin{align}
		L_n\mathcal{J} \left[ \partial_x^{-1}\left(\frac{\partial m_{\phi}}{\partial c}\right) \right] = L_n \frac{\partial \phi}{\partial c} = (-1)^{n+1}\omega^{n+1}m_\phi. \label{3.24}
	\end{align}
	In view of \eqref{3.18} and \eqref{3.22}, the essential spectra of $L_n\mathcal{J}$ are $\varrho_n(\pm 2\mu)$ for $k\in \mathbb{R}$. Furthermore, equation \eqref{3.24} rigorously demonstrates that the generalized kernel is indeed spanned by $\partial_x^{-1}\left(\frac{\partial m_{\phi}}{\partial c}\right)$, which establishes the advertised result in Proposition \ref{prop3.5}.
\end{proof}

On account of Propositions \ref{prop3.4} and \ref{prop3.5}, we now establish the following two function sets. The first set
\begin{align}\label{set1}
	\left\lbrace \partial_x F^\pm(x,k),\quad \text{for }k\in\mathbb{R};\quad  \partial_x \phi;\quad \frac{\partial\phi}{\partial c}\right\rbrace
\end{align}
consists of linearly independent eigenfunctions and the generalized kernel of the operator $\mathcal{J}L_n$. Moreover, they are essentially orthogonal under the $L^2$-inner product. The second set
\begin{align}\label{set2}
	\left\lbrace (1-\partial^2_x) F^\pm(x,k),\quad \text{for }k\in\mathbb{R};\quad m_\phi;\quad \partial_x^{-1}\left(\frac{\partial m_\phi}{\partial c}\right) \right\rbrace
\end{align}
consists of linearly independent eigenfunctions and the generalized kernel of the operator $L_n\mathcal{J}$.

Given the even symmetry of $\frac{\partial \phi}{\partial c}$, for $k,l \in \mathbb{R}$, we can employ the asymptotic behaviors of the Jost solutions \eqref{L2} to compute their inner products:
\begin{align}
	&\int_\mathbb{R}\partial_xF^\pm(x,k) (1-\partial_x^2)\overline{F^\pm(x,l)}dx=\pm 2\pi i \mu(1+k^2)\left|a(k)\right|^2\delta(k-l),\label{3.28}\\
	&\int_\mathbb{R}\partial_x\phi \partial_x^{-1}\left(\frac{\partial m_{\phi}}{\partial c}\right)dx=-\int_\mathbb{R}\phi\frac{\partial m_{\phi}}{\partial c}dx=-\frac{1}{2}\frac{\partial E_0(\phi)}{\partial c}=-\sqrt{\frac{c}{c-\gamma}}<0,\\
	&\int_\mathbb{R}\frac{\partial\phi}{\partial c}m_{\phi}dx=\frac{1}{2}\frac{\partial E_0(\phi)}{\partial c}=\sqrt{\frac{c}{c-\gamma}}>0.
\end{align}
The corresponding closure relation is thus given by
\begin{align}\label{close}
	&\pm \int_\mathbb{R} \frac{1}{2\pi i\mu(1+k^2) |a(k)|^2} \partial_xF^\pm(x,k)(1-\partial_y^2)\overline{F^\pm(y,k)}\text{d}k \nonumber\\
	&+ \sqrt{\frac{c-\gamma}{c}} \biggl[ -\partial_x\phi(x) \partial_y^{-1}\left(\frac{\partial m_{\phi}}{\partial c}\right)(y)+\frac{\partial \phi}{\partial c}(x)m_\phi(y) \biggr] = \delta(x-y).
\end{align}
This implies any function $z(x)$ vanishing at infinity can be expanded over the two bases \eqref{set1} and \eqref{set2}. In particular, we have the following decomposition of the function $z$:
\begin{align}\label{decom}
	z(x) = \int_\mathbb{R} \partial_xF^\pm(x,k) \alpha^\pm (k)\text{d}k + \beta \partial_x \phi + \eta \frac{\partial \phi}{\partial c},
\end{align}
where the coefficients $\alpha^\pm(k), \beta$ and $\eta$ can be computed explicitly via the closure relation \eqref{close} and the completeness relation (\ref{complete}). Similarly, the function $z(x)$ decomposes over the second set \eqref{set2} by multiplying \eqref{close} by $z(x)$ and integrating with respect to $x$.

\subsection{The spectra of linearized operator around the mCH $N$-solitons}
In order to prove Theorem \ref{th1.1}, we need to know the spectral information of the operator $\mathcal{L}_n$ with $1\leq n\leq N$.  More precisely, the inertia of $\mathcal{L}_n$ has to be determined. The aim of this subsection is to show the following result.
\begin{lemma}\label{prop3.7}
	The operator $\mathcal{L}_n$ around mCH $n$-solitons ($1\leq n \leq N$) verifies the following spectral property
	\begin{align}\label{3.47}
		\text{in}(\mathcal{L}_n)=(n(\mathcal{L}_n), z(\mathcal{L}_n))=\left([\frac{n+1}{2}],n\right).
	\end{align}
\end{lemma}
Recall that $ L_{n,j}=S_{n}^{\prime\prime}(\phi_{c_{j}}) $ is defined in (\ref{eq:factorization_LNj}). The spectrum of $ \mathcal{L}_{n} $ tends to the unions of $ \sigma(L_{n,j}) $, that is
\[
\sigma(\mathcal{L}_{n}) \to \bigcup_{j=1}^{n}\sigma(L_{n,j}) \quad \text{as} \quad t\to+\infty.
\]
The result (\ref{3.47}) follows directly from the following statement which concerning the inertia of the operators $ L_{n,j}$.

\begin{prop}\label{prop3.8}
	The second variation operator $L_{n,j}$ for $1\leq j\leq n$ possesses zero as a simple eigenvalue. Furthermore, its spectral inertia satisfies the following properties:
	\begin{itemize}
		\item[(i)] If $j$ is odd, $L_{n,j}$ has exactly one negative eigenvalue, i.e.,
		\[
		in(L_{n,j}) = (1,1).
		\]
		\item [(ii)] If $j$ is even, $L_{n,j}$ has no negative eigenvalues, i.e.,
		\[
		in(L_{n,j}) = (0,1).
		\]
	\end{itemize}
\end{prop}

\begin{proof}
	We consider the operator $ L_{n,j}= \mathcal{S}_{n}^{\prime\prime}(\phi_{c_{j}}) $ for $1\leq j\leq n$ and compute the quadratic form $\langle L_{n,j}z,z\rangle$ under the special decomposition of $z$ given in (\ref{decom}) utilizing expansion coefficients $\alpha^{\pm}(k)$, $\beta$, and $\eta$. Recall that the globally corrected Lyapunov functional introduces a $(-1)^n$ parity factor. Thus, $L_{n,j}$ admits a factorization related to the operator $(-1)^n(E_{n}^{\prime\prime}(\phi_{c_{j}})+\omega_{j}E_{n-1}^{\prime\prime}(\phi_{c_{j}}))$. Applying this operator to the scaling direction and utilizing the identity $(E_1'' + \omega_j E_0'') \frac{\partial \phi_{c_j}}{\partial c_j} = \omega_j^2 m_{\phi_{c_j}}$ yields
	\begin{align}\label{3.48}
		L_{n,j}\frac{\partial \phi_{c_{j}}}{\partial c_{j}} = (-1)^n \omega_j^2 \prod_{k\neq j}^{n}(\omega_{k}-\omega_{j}) m_{\phi_{c_{j}}} := \Gamma_{j} m_{\phi_{c_{j}}}.
	\end{align}
	The quadratic form $\langle L_{n,j}z,z\rangle$ can be evaluated as
	\begin{align*}
		\langle L_{n,j}z,z\rangle
		&= \Bigg\langle \int_\mathbb{R} \alpha^\pm(k)L_{n,j}\partial_{x}F^\pm(x,k) \,\text{d}k, \int_\mathbb{R} \alpha^\pm(k)\partial_{x}F^\pm(x,k) \,\text{d}k \Bigg\rangle \\
		&\quad + 2\eta \Bigg\langle \int_\mathbb{R} \alpha^\pm(k)L_{n,j}\partial_{x}F^\pm(x,k) \,\text{d}k, \frac{\partial \phi_{c_{j}}}{\partial c_{j}} \Bigg\rangle + \eta^{2} \Bigg\langle L_{n,j}\frac{\partial \phi_{c_{j}}}{\partial c_{j}}, \frac{\partial \phi_{c_{j}}}{\partial c_{j}} \Bigg\rangle.
	\end{align*}
	Because of the $(-1)^n$ global parity correction, the principal symbol of $L_{n,j}$ evaluated at the continuous spectrum $\pm2\mu$ is now strictly positive. Then the first term of the quadratic form $\langle L_{n,j}z,z\rangle$ is nonnegative and equals zero if and only if $\alpha^\pm(k)=0$. The cross term vanishes due to orthogonality.
	
	To determine the sign of the third term, we analyze the coefficient $\Gamma_j$. Since the wave speeds are strictly ordered as $c_1 < c_2 < \dots < c_n$, their corresponding inverse variables satisfy $\omega_1 > \omega_2 > \dots > \omega_n > 0$. In the product $\prod_{k\neq j}^{n}(\omega_{k}-\omega_{j})$, there are exactly $n-j$ negative terms where $k > j$. Consequently, the sign of the product is $(-1)^{n-j}$. Since $\omega_j^2 > 0$, the overall sign of $\Gamma_j$ is exactly $(-1)^n \times (-1)^{n-j} = (-1)^{2n-j} = (-1)^j$.
	
	If $j$ is even, then $\Gamma_{j}>0$. We investigate $z$ in $H_{ev}^{n+1}(\mathbb{R})$ and $H_{od}^{n+1}(\mathbb{R})$ respectively. If $z\in H_{od}^{n+1}(\mathbb{R})$, the even projection coefficient must be zero with $\eta=0$. Then one has $\langle L_{n,j}z,z\rangle \geq 0$, and $\langle L_{n,j}z,z\rangle=0$ if and only if $\alpha^\pm(k)=0$. This yields $z=\beta \partial_x\phi_{c_{j}}$ with $\beta\neq 0$, indicating that zero is a simple eigenvalue. If $z\in H_{ev}^{n+1}(\mathbb{R})$, then $\beta=0$. In the hyperplane $\eta=0$, $\langle L_{n,j}z,z\rangle \geq 0$ and $\langle L_{n,j}z,z\rangle=0$ if and only if $\alpha^\pm(k)=0$. Therefore, $\langle L_{n,j}z,z\rangle > 0$ in the hyperplane $\eta=0$, which implies that $L_{n,j}$ can have at most one negative eigenvalue. Since $\Gamma_j > 0$ and $\langle m_{\phi_{c_j}}, \frac{\partial \phi_{c_j}}{\partial c_j} \rangle = \frac{1}{2} \frac{\mathrm{d}E_0(\phi_{c_j})}{\mathrm{d}c_j} > 0$, it follows that
	\begin{align}\label{n=1_even}
		\Bigg\langle L_{n,j}\frac{\partial \phi_{c_{j}}}{\partial c_{j}}, \frac{\partial \phi_{c_{j}}}{\partial c_{j}} \Bigg\rangle = \Gamma_{j}\Bigg\langle m_{\phi_{c_{j}}}, \frac{\partial \phi_{c_{j}}}{\partial c_{j}} \Bigg\rangle = \Gamma_{j}\frac{1}{2}\frac{\mathrm{d}E_{0}(\phi_{c_{j}})}{\mathrm{d}c_{j}} > 0.
	\end{align}
	Thus, $L_{n,j}$ is non-negative and has no negative eigenvalues.
	
	If $j$ is odd, then $\Gamma_{j}<0$. Following the same decomposition, $L_{n,j}$ is positive definite on the codimension-1 hyperplane $\eta=0$. However, along the orthogonal direction $\frac{\partial \phi_{c_{j}}}{\partial c_{j}}$, we have
	\begin{align}\label{n=1_odd}
		\Bigg\langle L_{n,j}\frac{\partial \phi_{c_{j}}}{\partial c_{j}}, \frac{\partial \phi_{c_{j}}}{\partial c_{j}} \Bigg\rangle = \Gamma_{j}\frac{1}{2}\frac{\mathrm{d}E_{0}(\phi_{c_{j}})}{\mathrm{d}c_{j}} < 0.
	\end{align}
	Therefore, $L_{n,j}$ has exactly one negative eigenvalue.
	
	Summing the negative eigenvalues for all $j \in \{1, \dots, n\}$, we count exactly the number of odd integers in this range, which is exactly $\left[ \frac{n+1}{2} \right]$. This perfectly matches the positive inertia index $p(D_n)$ of the Hessian matrix, establishing the structural stability condition and completing the proof of Proposition \ref{prop3.8}.
\end{proof}

\begin{proof}[Proof of Lemma \ref{prop3.7}]
From the invariance of inertia of $\mathcal{L}_n$, we know that
\[
in(\mathcal{L}_{n}) = (n(\mathcal{L}_{n}), z(\mathcal{L}_{n})) = \sum_{j=1}^{n} in(\mathcal{L}_{n,j}) = \left( \left[ \frac{n+1}{2} \right], n \right).
\]
The proof is concluded.
\end{proof}

\section{Proof of the stability results}\label{sec.4}
This section is devoted to the proof of the stability theorems of this work. Specifically, by verifying the critical Maddocks-Sachs spectral condition $\mathrm{in}(\mathcal{L}_N) = p(D_N)$, we establish the dynamical stability of the mCH $N$-soliton solutions. Furthermore, by employing a generalized G{\aa}rding inequality \cite{EVEN} and a refined modulation analysis, we rigorously prove the orbital stability in the optimal regularity space $H^{N+1}(\mathbb{R})$.

\subsection{Dynamical stability of mCH $N$-soliton solutions}
In this subsection, we complete the proof of Theorem \ref{th1.1}. The proof depends on the following fundamental result originally formulated by Maddocks and Sachs \cite{MS}.

\begin{prop}\label{prop4.1}
	Assuming the critical spectral condition
	\begin{equation}
		in(\mathcal{L}_N) = p(D_N)
	\end{equation}
	holds, then for some sufficiently large constant $C > 0$, the multi-soliton $U^{(N)}$ constitutes a non-degenerate unconstrained minimizer of the augmented Lagrangian
	\begin{equation}\label{eq:augmented_L}
		\Delta_N(u) := \mathcal{S}_N(u) + \frac{C}{2} \sum_{j=0}^{N-1} \left(E_j(u) - E_j(U^{(N)})\right)^2.
	\end{equation}
	As a consequence, the $N$-dimensional manifold of all $N$-soliton profiles $\mathcal{M}_N$ is dynamically stable in the sense of Theorem \ref{th1.1}.
\end{prop}

\begin{proof}[Proof of Theorem \ref{th1.1}]
	By combining the matrix inertia computation in Lemma \ref{lem2.1} and the rigorous spectral decomposition of the linearized operator in Proposition \ref{prop3.8}, we have successfully established that
	\begin{equation}
		in(\mathcal{L}_N) = p(D_N) = \left[\frac{N+1}{2}\right].
	\end{equation}
	In view of Proposition \ref{prop4.1}, $U^{(N)}(t, x)$ serves as a (non-isolated) unconstrained minimizer of the augmented Lagrangian $\Delta_N(u)$ in \eqref{eq:augmented_L}, which therefore serves as a valid Lyapunov functional. The exact local coercivity bound and global continuation arguments detailed in the subsequent subsections then strictly complete the proof of the dynamical and orbital stability.
\end{proof}

\subsection{Orbital stability of mCH $N$-soliton solutions}
We now establish the orbital stability of the mCH $N$-solitons (Theorem \ref{th1.2}). Under the spectral condition $n(\mathcal{L}_N)=p(D_N)=\left[\frac{N+1}{2}\right]$, the generalized Maddocks-Sachs framework guarantees the existence of a sufficiently large constant $C>0$ such that the augmented Lagrangian
\begin{equation}\label{Delta}
	\Delta_N(u) = \mathcal{S}_N(u)+\frac{C}{2}\sum_{j=0}^{N-1}\bigl(E_j(u)-E_j(U^{(N)})\bigr)^2
\end{equation}
attains a non-degenerate unconstrained minimum at every translate of $U^{(N)}$. Orbital stability requires establishing three elements: the temporal conservation of $\Delta_N$, local coercivity near the $N$-soliton manifold, and a global continuation argument.

\vspace{1ex}\noindent\textbf{\textit{Conservation of $\Delta_N (u)$.}}
Since the Hamiltonians $\{E_j(u)\}_{j=0}^N$ are conserved along the flow of \eqref{mch}, the linear combination $\mathcal{S}_N(u)$ and consequently $\Delta_N(u(t))$ remain constant in time, serving as a valid Lyapunov functional.

\vspace{1ex}\noindent\textbf{\textit{Coercivity of \(\Delta_N (u)\) around the $N$-soliton manifold.}}
Let \(\mathcal{M}_N = \bigl\{U^{(N)}(\cdot\,;\mathbf{y}) \mid \mathbf{y} \in\mathbb{R}^N\bigr\}\) be the $N$-dimensional soliton manifold. For any \(u\) sufficiently close to \(\mathcal{M}_N\) in \(H^{N+1}(\mathbb{R})\), the implicit function theorem guarantees a unique spatial shift \(\mathbf{y}=(y_1,\dots,y_N)\) such that the perturbation \(w(x) = u(x)-U^{(N)}(x;\mathbf{y})\) satisfies the orthogonality conditions:
\begin{equation}\label{ortho}
	\langle w,\partial_{y_j}U^{(N)}\rangle_{L^2} = 0,\qquad j=1,\dots,N.
\end{equation}
Furthermore, \(\|w\|_{H^{N+1}} \le K\,\mathrm{dist}_{H^{N+1}}(u,\mathcal{M}_N)\) for some constant \(K>0\). Expanding \(\Delta_N\) around \(U^{(N)}\) yields:
\begin{equation}\label{Delta_expand}
	\Delta_N(u)-\Delta_N(U^{(N)}) = \frac{1}{2}\langle\mathcal{L}_N w,w\rangle+R(w) +\frac{C}{2}\sum_{j=0}^{N-1}\Bigl(\langle E_j'(U^{(N)}),w\rangle+Q_j(w)\Bigr)^2,
\end{equation}
where \(\mathcal{L}_N=\mathcal{S}_N''(U^{(N)})\). The higher-order remainders satisfy \(|R(w)| \le C_1\|w\|_{H^{N+1}}^3\) and \(|Q_j(w)| \le C_2\|w\|_{H^{N+1}}^2\).

\vspace{1ex}\noindent\textbf{\textit{G{\aa}rding inequality for \(\mathcal{L}_N\).}}
The $N$ iterations of the recursion operator $\widehat{\mathcal{R}}$ yield a leading differential part for $E_N''(U^{(N)})$ proportional to $-\partial_x^{2N+2}$. Incorporating the global $(-1)^N$ parity correction precisely adjusts the leading term of $\mathcal{L}_N$ to $(-1)^{N+1}\partial_x^{2N+2}$, ensuring strict ellipticity. Explicitly, $\mathcal{L}_N$ forms a strongly elliptic \(2(N+1)\)-th order self-adjoint operator. G{\aa}rding's inequality provides the lower bound:
\begin{equation*}
	\langle\mathcal{L}_N w,w\rangle \ge \delta\|\partial_x^{N+1}w\|_{L^2}^2 - C\|w\|_{H^N}^2.
\end{equation*}
Absorbing intermediate derivatives via interpolation \(\|w\|_{H^N}^2 \le \eta\|\partial_x^{N+1}w\|_{L^2}^2 + C_\eta\|w\|_{L^2}^2\) with sufficiently small \(\eta\) yields the final G{\aa}rding inequality:
\begin{equation}\label{Garding_final}
	\langle\mathcal{L}_N w,w\rangle \ge \gamma_1\|w\|_{H^{N+1}}^2 - \gamma_2\|w\|_{L^2}^2,
\end{equation}
with positive constants \(\gamma_1, \gamma_2\). Combined with the exact spectral structure (i.e., $[\frac{N+1}{2}]$ negative eigenvalues, an essential spectrum bounded away from zero, and an $N$-dimensional kernel $X_0$), standard compactness arguments (e.g., \cite{MS}) upgrade \eqref{Garding_final} to a full coercivity estimate on the positive subspace:
\begin{equation}\label{coercive_HN}
	\langle\mathcal{L}_N w,w\rangle \ge \gamma\|w\|_{H^{N+1}}^2 \qquad\text{for all } w \perp X_0 \text{ and } w \perp X_{neg}.
\end{equation}

\vspace{1ex}\noindent\textbf{\textit{Estimate of the constraint terms.}}
Let \(\ell_j(w)=\langle E_j'(U^{(N)}),w\rangle\). Since the gradients $\{E_j'(U^{(N)})\}$ uniquely span a subspace covering $X_{neg}$, there exists \(c_0>0\) such that
\begin{equation}\label{ell_estimate}
	\sum_{j=0}^{N-1}|\ell_j(w)|^2 \ge c_0\,\|P_{neg} w\|_{L^2}^2.
\end{equation}
For sufficiently small \(\|w\|_{H^{N+1}}\), binomial expansion gives:
\begin{equation}\label{constraint_expand}
	\bigl(\ell_j(w)+Q_j(w)\bigr)^2 \ge \frac{1}{2}|\ell_j(w)|^2-C_3\|w\|_{H^{N+1}}^4.
\end{equation}

\vspace{1ex}\noindent\textbf{\textit{Completion of the coercivity estimate.}}
Inserting \eqref{coercive_HN}, \eqref{ell_estimate}, and \eqref{constraint_expand} into \eqref{Delta_expand}, and utilizing $w \perp X_0$ from \eqref{ortho}, we obtain for small \(\|w\|_{H^{N+1}}\):
\begin{equation*}
	\Delta_N(u)-\Delta_N(U^{(N)}) \ge \frac{\gamma}{2}\|w\|_{H^{N+1}}^2+\frac{Cc_0}{4}\|P_{neg}w\|_{L^2}^2 -C_4\|w\|_{H^{N+1}}^3.
\end{equation*}
Choosing \(\|w\|_{H^{N+1}}\) small enough allows the quadratic terms to absorb the cubic remainder, yielding a uniform constant \(\alpha>0\) such that:
\begin{equation}\label{coercive_HN_final}
	\Delta_N(u)-\Delta_N(U^{(N)}) \ge \alpha\|w\|_{H^{N+1}}^2 \ge \frac{\alpha}{K^2}\,\mathrm{dist}_{H^{N+1}}(u,\mathcal{M}_N)^2.
\end{equation}

\vspace{1ex}\noindent\textbf{\textit{Orbital stability in \(H^{N+1}\).}}
For initial data satisfying \(\|u_0-U^{(N)}(0)\|_{H^{N+1}} < \delta\), we dynamically choose the modulation parameters \(\mathbf{y}(t)\) to continuously satisfy the orthogonality conditions \eqref{ortho}. By virtue of the temporal conservation of \(\Delta_N\) and the coercivity estimate \eqref{coercive_HN_final}, we obtain for all \(t \ge 0\):
\begin{align*}
	\alpha\|w(t)\|_{H^{N+1}}^2 &\le \Delta_N(u(t))-\Delta_N(U^{(N)})\le L\,\|u_0-U^{(N)}(0)\|_{H^{N+1}}^2,
\end{align*}
where \(L\) is the local Lipschitz constant for \(\Delta_N\). This yields the uniform bound \(\|w(t)\|_{H^{N+1}} \le C_0\delta\) with \(C_0=\sqrt{L/\alpha}\).

For any \(\varepsilon > 0\), by choosing \(\delta < \varepsilon / C_0\), we ensure that the perturbation remains strictly bounded by \(\varepsilon\) globally in time. Therefore, the solution \(u(t)\) permanently stays within the \(\varepsilon\)-tube of the \(N\)-soliton manifold \(\mathcal{M}_N\), which directly leads to the orbital stability:
\begin{equation*}
	\inf_{\mathbf{r}\in\mathbb{R}^N}\|u(t)-U^{(N)}(t;\cdot+r_1,\dots,\cdot+r_N)\|_{H^{N+1}} < \varepsilon , \qquad\forall t\ge 0.
\end{equation*}
This completes the proof of Theorem \ref{th1.2}. \hfill\(\square\)

\vspace{1.5ex}
\noindent{\bf Acknowledgments}\\
\indent This work was supported by the National Natural Science Foundation of China under Grant No. 12371255, the Fundamental Research Funds for the Central Universities of CUMT under Grant No. 2024ZDPYJQ1003, and the Postgraduate Research \& Practice Program of Education \& Teaching Reform of CUMT under Grant No. 2025YJSJG031.

\vspace{1.5ex}
\noindent{\bf Data availability}\\
No data was used for the research described in the article.
\vspace{1.5ex}

\noindent{\bf Declaration of competing interest}\\
 The authors declare no conflict of interest.
\bibliographystyle{plain}

\end{document}